\documentclass[preprint,12pt,english]{elsarticle}

\usepackage[margin=2cm]{geometry}
\usepackage{mathrsfs}
\usepackage{amsfonts}
\usepackage{dsfont}
\usepackage{multirow}
\usepackage{bbold}
\usepackage{amssymb}
\usepackage{float}
\usepackage{placeins}
\usepackage[utf8]{inputenc}
\usepackage{hyperref}
\usepackage{amsthm}
\usepackage{tikz}
\usetikzlibrary{arrows.meta}
\usetikzlibrary{backgrounds,arrows, shapes,decorations.markings,positioning}
\usetikzlibrary{plotmarks,calc,fadings,decorations.pathreplacing,decorations.pathmorphing}
\tikzset{%
  >=latex,
  inner sep=0pt,%
  outer sep=2pt,%
  mark coordinate/.style={inner sep=0pt,outer sep=0pt,minimum size=3pt,
    fill=black,circle}%
}

\tikzset{-dot-/.style={decoration={
  markings,
  mark=at position #1 with {\fill circle (2.0pt);}},postaction={decorate}}} %%% in this line added a ;

\usepackage{pgfplots}
\usepackage{pgfplotstable}
\pgfplotsset{compat=1.13}
\usepackage{amsmath}
\usepackage{subfigure}
\usepackage{epstopdf}
\usepackage{graphicx}
\usepackage{color}
\usepackage{cleveref}

\newcommand{\Om}{\Omega}
\newcommand{\R}{\mathbb{R}}

\newcommand{\pa}{\partial}
\newcommand{\diff}[1]{{\mathrm{d}{#1}}}
\newcommand{\D}{\mathcal{D}}

\newtheorem{theorem}{Theorem}

\newtheorem{proposition}[theorem]{Proposition}

\newtheorem{remark}[theorem]{Remark}

\journal{Elsevier}

\begin{document}

\begin{frontmatter}


\title{Minimization-based polynomial corrections for \\ high-order curved boundaries on fixed and moving domains: \\
assessment on finite volume and discontinuous Galerkin schemes}

%use optional labels to link authors explicitly to addresses:
\author[lab1]{Mirco Ciallella\corref{cor1}}
\author[lab2,dmi]{Walter Boscheri}

\address[lab1]{Universit\'e Paris Cit\'e and Sorbonne Universit\'e, CNRS, LJLL, F-75005 Paris, France}
\address[lab2]{Laboratoire de Math\'ematiques, Universit\'e Savoie Mont Blanc, CNRS, UMR 5127, Chamb\'ery, France}
\address[dmi]{Department of Mathematics and Computer Science, University of Ferrara, Ferrara, Italy}

\cortext[cor1]{Corresponding author (\href{mailto:mirco.ciallella@u-pariscite.fr}{mirco.ciallella@u-pariscite.fr})}

%% Abstract
\begin{abstract}
In this work, we present two novel strategies to impose high-order boundary conditions on fixed and moving curved domains, approximated with piecewise affine triangulations. 
Achieving high-order accuracy on curved domains consists in tackling both the PDE discretization error and the geometrical error simultaneously. 
Although the former can be reduced by employing high-order numerical methods such as finite volume and discontinuous Galerkin, the latter requires either a high-order parametrization of the physical domain or a consistent approximation of the boundary conditions. 
Minimization-based approaches like the Reconstruction for Off-site Data (ROD) method allow one to skip the construction of high-order curvilinear meshes by defining high-order consistent boundary conditions on a computational boundary that does not match the physical one. 
Therefore, the ROD approach is able to mitigate the second-order geometrical error, introduced by the approximation of the domain through simplicial elements, by retrieving a modified polynomial in each boundary cell which is as close as possible to the internal one while enforcing exactly the boundary conditions on a set of points located on the physical boundary. However, the standard ROD approach requires the inversion of a linear system in each boundary cell that introduces a computational cost, which grows with the polynomial degree and the refinement level of the mesh.
Inspired by a recent one-dimensional analysis, we show here that ROD-type approaches can be recast as simple polynomial corrections which can be applied without the inversion of any linear system. This greatly simplifies the development of minimization-based boundary treatments within various numerical frameworks, and reduces the computational cost related to the boundary treatment especially when dealing with complex applications.    
To prove the wide applicability of our novel strategy, we propose to develop it in a Runge-Kutta discontinuous Galerkin framework and in an ADER arbitrary-Lagrangian-Eulerian finite volume framework for the simulation of compressible flows on fixed and moving domains, respectively. Several numerical experiments with Dirichlet-type and slip-wall boundary conditions are presented with convergence analysis up to fifth order in both 2D and 3D.
\end{abstract}


% \begin{highlights}
% \end{highlights}

%% Keywords

\end{frontmatter}

\section{Introduction}\label{section:introduction}

Achieving high-order of accuracy on curved domains is a real challenge in computational physics.
This is related to the fact that the total discretization error is given by the error produced by the numerical method used to approximate the governing partial differential equation (PDE) and by the geometrical error induced by the domain approximation. Although high-order methods have shown a strong potential in obtaining lower errors on coarser grids \cite{Wang_et_al_ijnmf13,cockburn1998runge,jiang1996efficient,dumbser2008unified}, an appropriate treatment of curved boundaries is crucial to maintain the expected precision. 

In high-fidelity simulations, a common strategy to handle curved geometries is the isoparametric approach \cite{bassi1997high,zienkiewic}, where the same polynomial degree is used to approximate both the solution and the physical domain. This practice ensures optimal convergence rates and enhances accuracy on coarse meshes, particularly in the presence of curved boundaries. 
Existing strategies for curved geometry representation include classical polynomial mappings as well as more recent approaches based on rational B-splines and NURBS, as employed in isogeometric analysis~\cite{nurbs0,PEZZANO2021110093}. In all cases, however, the quality of the curved mesh plays a crucial role in the overall accuracy and robustness of the simulation. To simplify this process, some methods propose to curve straight faced meshes~\cite{dey2001towards,luo2004automatic,sahni2010curved}.
However, despite significant progress in recent years, the generation of high quality curvilinear meshes and the simulation of problems defined on them remains a complex task, still subject of intense research. The mesh generation itself introduces further difficulties, as advanced techniques are often required to guarantee mesh validity and element quality. On the contrary, linear mesh generation is now a well-established technology, thanks to decades of development and its widespread use across research and industry. 
The same considerations hold when dealing with moving curved domains. Arbitrary-Lagrangian-Eulerian (ALE) formulations \cite{kucharik2011hybrid,shashkov2008closure,boscheri2013arbitrary} are particularly well-suited to follow moving boundaries, as the mesh can adapt to the motion of the geometry while allowing a certain flexibility compared to pure Lagrangian schemes \cite{munz1994godunov,carre2009cell,despres2003symmetrization,atallah2024weak,kincl2025semi,Maire2009}. However, achieving high-order accuracy in this context requires ALE methods on curved moving meshes \cite{boscheri2016high}, which can become extremely cumbersome from the computational viewpoint. The need to continuously update the mesh, combined with the complexities of curved element handling, significantly increases the algorithmic and computational burden, especially for large-scale or long-time simulations.

A more computationally efficient alternative consists in improving the accuracy of boundary conditions on simplicial meshes, where the geometry is discretized using simple linear or planar elements such as triangles or tetrahedra. In this setting, high-order precision is achieved not by refining the mesh, but by suitably modifying the boundary conditions. This approach, however, demands a careful treatment of the local geometric features near the boundary. Early contributions in this direction can be found in \cite{WangSun,krivodonova2006high}, where curvature-corrected boundary conditions were introduced and applied to high-order schemes.
More recent works in the context of high-order corrections for general boundary corrections can be found in \cite{costa2018very,costa2019very,santos2024very,ciallella2023shifted,ciallella2024very,boscheri2025high}.
Among these, the Reconstruction for Off-site Data (ROD) method was developed to improve the consistency of boundary conditions on curved domains discretized with simplicial elements in a finite volume (FV) framework \cite{costa2018very}. Its goal was to perform a finite volume high-order reconstruction at boundary elements with a modified polynomial that is aware of the exact boundary conditions. More precisely, this was formulated by defining a constrained least-squares problem where Lagrange multipliers are used to enforce the exact boundary conditions. For discontinous Galerkin (DG) methods, the ROD approach consists in finding a polynomial that minimize the distance with the internally extrapolated one while enforcing exactly the boundary conditions on the curved boundary \cite{santos2024very,ciallella2024very}. By employing the same polynomial basis used for the internal polynomial also for the modified one, the DG-ROD formulation appears to have a simplified linear system to be inverted. However, the method still involves the solution of this linear system which needs to be inverted at every time step and in each boundary cell for moving boundaries.
Inspired by a one-dimensional analysis \cite{ciallella2025minimization}, where only one boundary constraint is imposed, we show that the ROD method can be recast as a simple polynomial correction, which relies upon the evaluation of basis functions at the real boundary point, in the spirit of the shifted boundary polynomial correction \cite{ciallella2023shifted,boscheri2025high}. 

The core idea of this work is that the one-dimensional polynomial correction can be extended to a wide range of computational frameworks, requiring only a polynomial representation of the internal solution, and to multiple dimensions by applying the correction pointwise in the boundary integrals at quadrature points, while always preserving high-order consistency. In particular, the paper will consider two polynomial corrections based on different distance functions in the minimization problem: we will refer to the one based on the squared Euclidean norm as the ROD-E correction, and to the one based on the $L^2$ norm as the ROD-$L^2$ correction.
It should be noted that, although in one dimension the ROD-E correction is a reformulation of the original ROD method, in multiple dimensions, where the correction is imposed pointwise along the boundary mapping, a different scheme is obtained. Through a computational complexity analysis, we show that the novel corrections in multiple dimensions are significantly cheaper than the original ROD formulation, which involves multiple constraints and thus requires the inversion of the constraint matrix.
To prove the effectiveness and robustness of these methods, we validate their high-order convergence properties by performing simulations with Runge-Kutta (RK) DG methods on curved fixed domains \cite{cockburn1998runge} and with ADER-ALE FV methods on moving curved domains \cite{boscheri2013arbitrary}. As mentioned above, the only requirement is a polynomial representation of the internal solution, and the method can be applied straightforwardly. In RK-DG, the polynomial correction exploits the internal polynomial of the boundary cell. In the FV context, it simply uses the high-order reconstruction, expressed in polynomial form. For ADER methods, both in the FV and DG variants, the correction can be applied using the high-order space-time predictor. The same applies when ADER methods are used within the ALE framework.

The rest of the paper is organized as follows.
In \cref{section:model}, we present the governing equations used in this work: the nonlinear Euler equations for compressible gas dynamics.
In \cref{section:ALE-FV-DG}, we present the two computational frameworks used to validate the new methods. In particular, in \cref{sec:DG}, the RK-DG framework is presented in the context of fixed domains. While, in \cref{sec:FV}, the ADER-ALE FV framework is presented in the context of moving domains.
\cref{section:ROD} presents the high-order boundary conditions used in this work. This section discusses first the literature of the ROD method and the original approach in the context of DG. Afterwards, it presents the novel polynomial corrections and how to implement them in general computational frameworks. Numerical experiments are presented in \cref{section:ResultsDG} and in \cref{section:ResultsFV} for the RK-DG and the ADER-ALE FV frameworks, respectively. Conclusions and future perspectives are given in \cref{section:Conclusions}.

\section{Governing equations}\label{section:model}

We consider the numerical approximation of solutions of the Euler equations for compressible gas dynamics in $d$ space dimensions, written as:

\begin{equation}\label{eq:euler0}
\partial_t \mathbf{q}+\nabla\cdot\mathbf{F}\left(\mathbf{q}\right) =0, \quad \text{on}\quad \Omega_T=\Omega\times[0,T]\subset\mathbb{R}^d\times\mathbb{R}^+,
\end{equation}
where $\mathbf{q}$ is the vector of conserved variables and $\mathbf{F}$ the nonlinear flux, respectively defined as:
\begin{equation}\label{eq:euler0a}
\mathbf{q}=
\begin{bmatrix}
\rho \\ \rho \mathbf{U}  \\ \rho E
\end{bmatrix},\qquad
\mathbf{F}(\mathbf{q})= 
\begin{bmatrix}
\rho \mathbf{U}\\ \rho \mathbf{U} \otimes \mathbf{U} + p \mathbb{I} \\ \rho H \mathbf{U}
\end{bmatrix}, 
\end{equation}
with $\rho$ the mass density, $\mathbf{U}=[U_1,\ldots,U_d]^T$ the velocity, $p$ the pressure, and $E = e + \mathbf{U}\cdot\mathbf{U}/2$ the specific total energy, where $e$ is the specific internal energy. The total specific enthalpy is $H = h + \mathbf{U}\cdot\mathbf{U}/2$, with $h = e + p/\rho$ the specific enthalpy. For simplicity, we adopt the classical perfect gas equation of state:
\begin{equation}\label{eq:EOS}
p = (\gamma-1)\rho e,
\end{equation}
with $\gamma$ the constant ratio of specific heats. Herein, the ratio $\gamma=1.4$ is assumed except when specified otherwise.

To perform convergence analysis, we employ the method of manufactured solutions, which requires the introduction of an additional source term in \cref{eq:euler0}. The modified system reads:

\begin{equation}\label{eq:euler1}
\partial_t \mathbf{q}+\nabla\cdot\mathbf{F}\left(\mathbf{q}\right) = \mathbf{s}(\mathbf{q}).
\end{equation}

This source term $\mathbf{s}$ will be taken into account in the discretizations presented in \cref{section:ALE-FV-DG}.

\section{Arbitrary high-order DG and FV methods on fixed and moving meshes}\label{section:ALE-FV-DG}

In this section, we present the numerical schemes used to solve the governing equations on fixed and moving meshes. We consider both Discontinuous Galerkin (DG) and Finite Volume (FV) methods, which are widely used for the numerical approximation of hyperbolic conservation laws. The main focus of this paper is on the treatment of high-order boundary conditions when the curved domain is discretized with linear meshes.
The novel boundary treatments, proposed in \cref{subsection:ROD1Dcorrection}, are applied to both DG and FV methods to impose high-order consistent Dirichlet-type and slip-wall boundary conditions even on moving domains, thus overcoming the second-order accuracy limitation of the simplicial elements. 

\subsection{Runge-Kutta Discontinuous Galerkin (RK-DG) method on fixed meshes}\label{sec:DG}

In this section, we recall the high-order RK-DG framework in the context of fixed domains.
The spatial domain $\Omega$ is discretized into $\Omega_h$ using a tessellation $\mathscr{T}$ composed of $\mathscr{N}$ non-overlapping simplicial elements (triangles in two dimensions, tetrahedra in three dimensions). A generic element is denoted by $T_i$, and we define $\Omega_h = \bigcup_{i=1}^{\mathscr{N}} T_i$. In general, $\Omega_h \neq \Omega$, and consequently $\partial\Omega_h \neq \partial\Omega$ for most approximations, including conformal ones, except for very simple geometries with no curvature or isogeometric approaches~\cite{nurbs0}.

The numerical solution $\mathbf{q}$ is approximated by $\mathbf{q}_h$, which belongs to a space of piecewise polynomials within each element $T_i$ and is discontinuous across element interfaces. Within each element $T_i$, we write
\begin{equation}\label{eq:DGbasis}
\mathbf{q}_h(\mathbf{x},t) = \sum_{k=1}^{D} \hat{\mathbf{q}}_{k}(t)\,\psi_{k}(\mathbf{x}),\qquad \mathbf{x}\in T_i. 
\end{equation}
Let $\mathbb{P}_p$ denote the space of polynomials of total degree at most $p$, and let $\{\psi_{k}\}_{k=1}^D$ be a basis of this space.
On simplicial elements, the number of degrees of freedom $D$ is given by $D = \prod_{\ell=1}^{d} (p+\ell)/\ell$. The representation~\eqref{eq:DGbasis} reduces to a first-order FV scheme when $p=0$.

\subsubsection{Discontinuous Galerkin discretization in space}

The elemental semi-discrete discontinuous Galerkin (DG) weak formulation is obtained by projecting each component of the governing equations onto the basis functions and integrating by parts~\cite{cockburn1998runge,bassi1997high}:
\begin{equation}\label{eq:DGweak semidiscrete 2}
\int_{T_i} \psi_j \, \frac{\mathrm{d} \mathbf{q}_h}{\mathrm{d} t} \, \mathrm{d}\mathbf{x} + \int_{\partial T_i} \psi_j \, \hat{F}(\mathbf{q}_h^-,\mathbf{q}_h^+,\mathbf{n}) \, \mathrm{d}S - \int_{T_i} \nabla\psi_j \cdot \mathbf{F}(\mathbf{q}_h) \, \mathrm{d}\mathbf{x} = \int_{T_i} \psi_j \, \mathbf{s}(\mathbf{q}_h) \, \mathrm{d}\mathbf{x}, 
\end{equation}
where $\hat{F}(\mathbf{q}_h^-,\mathbf{q}_h^+,\mathbf{n})$ is a consistent numerical flux that depends on the internal state $\mathbf{q}_h^-$, the neighbouring state $\mathbf{q}_h^+$, and the face normal $\mathbf{n}$. A consistent flux is a Lipschitz continuous function of its arguments and satisfies
\begin{equation}\label{eq:constistencyDG}
\hat{F}(\mathbf{q},\mathbf{q},\mathbf{n}) = \mathbf{F}(\mathbf{q})\cdot\mathbf{n}.
\end{equation}
In this work we employ a simple and robust Rusanov-type flux:
\begin{equation}\label{eq:rusanov flux}
\hat{F}(\mathbf{q}_h^-,\mathbf{q}_h^+,\mathbf{n}) = \frac{1}{2}\left(\mathbf{F}(\mathbf{q}_h^+) + \mathbf{F}(\mathbf{q}_h^-)\right)\cdot\mathbf{n} - \frac{1}{2}\lambda_{max}\left(\mathbf{q}_h^+ - \mathbf{q}_h^-\right),
\end{equation}
where $\lambda_{max}$ is the maximum eigenvalue of the normal flux Jacobians $A_{\mathbf{n}}(\mathbf{q}_h^+)$ and $A_{\mathbf{n}}(\mathbf{q}_h^-)$.

\subsubsection{Runge-Kutta time integration}

For fixed domain computations, we adopt a classical method of lines approach. Assembling all contributions from~\eqref{eq:DGweak semidiscrete 2}, we obtain in each element $T_i$ a system of ordinary differential equations~\cite{bassi1997high}:
\begin{equation}\label{eq:ODE semidiscrete}
\frac{\mathrm{d}\hat{\mathbf{q}}}{\mathrm{d}t} = M^{-1} \mathcal{K}(\hat{\mathbf{q}}),
\end{equation}
where the array $\hat{\mathbf{q}}$ contains the associated degrees of freedom, $\mathcal{K}$ is the array of size $D$ comprising the second, third, and fourth integrals in~\eqref{eq:DGweak semidiscrete 2}, and $M$ denotes the elemental mass matrix of dimension $D \times D$ given by
\begin{equation}\label{eq:massM}
[M]_{jk} = \int_{T_i} \psi_j \psi_k \, \mathrm{d}\mathbf{x}.
\end{equation}
We integrate~\eqref{eq:ODE semidiscrete} using classical Runge-Kutta methods~\cite{butcher2000numerical,gottlieb2001strong}.

\subsection{ADER-ALE Finite Volume (FV) method on moving meshes}\label{sec:FV}

In this section, we recall the ADER-ALE FV framework to simulate hyperbolic conservation laws in moving domains. 
The time-dependent computational domain $\Om^n =\Om_h(t^n)$ is discretized at the time level $t^n$  by non-overlapping simplicial elements $T_i^n$.
For conciseness, the description of the method focuses on 2D triangles, 
although the algorithm is also implemented and validated for 3D tetrahedrons. 
Additional details on the implementation are provided in \cite{boscheri2014direct}.
As also discussed in \cref{sec:DG}, in general, the moving tessellation does not coincide with the exact moving computational domain, except for very simple geometries with no curvature.

Unlike Eulerian approaches, where the mesh remains fixed, Lagrangian methods allow the computational grid to deform as the solution evolves in time. 
To simplify the development of the method, it is convenient to introduce a local reference coordinate system 
$\boldsymbol{\xi}=(\xi,\eta)$ with $(\xi,\eta)\in[0,1]$, where the reference element $T_e$ is defined. 
The reference element $T_e$ is a triangle with vertices located at $\boldsymbol{\xi}_{1,e} = (0,0)$, $\boldsymbol{\xi}_{2,e} = (1,0)$, and $\boldsymbol{\xi}_{3,e} = (0,1)$.
Each simplicial element $T^n_i$ at the current time $t_n$ is mapped from the reference 
system $(\xi, \eta)$ to the physical system $(x, y)$ by the following transformation: 
\begin{equation}\label{eq:referencesystem}
\begin{split}
  x &= X^n_{1,i} + (X^n_{2,i} - X^n_{1,i})\xi + (X^n_{3,i} - X^n_{1,i})\eta, \\
  y &= Y^n_{1,i} + (Y^n_{2,i} - Y^n_{1,i})\xi + (Y^n_{3,i} - Y^n_{1,i})\eta,
\end{split}
\end{equation}
where $\mathbf{X}^n_{k,i}=(X^n_{k,i},Y^n_{k,i})$ are the spatial coordinates of the $k$-th vertex 
of triangle $T^n_i$.   

In FV methods, the numerical solution within each element $T^n_i$ is described by spatial cell averages, which are defined by the following integral, approximated with consistent quadrature formulas:
\begin{equation}\label{eq:cellaverage}
  \mathbf{q}^n_i = \frac{1}{|T^n_i|} \int_{T^n_i} \mathbf{q}(\mathbf{x},t^n) \, \diff{\mathbf{x}},
\end{equation}
where $|T^n_i|$ is the volume of triangle $T^n_i$. 
High-order precision is obtained by reconstructing the solution at element interfaces with piecewise high-order polynomials 
\begin{equation}\label{eq:FVbasis}
\mathbf{w}_h(\mathbf{x},t^n) = \sum_{k=1}^{D} \hat{\mathbf{w}}_{k}(t^n)\,\psi_{k}(\mathbf{x}),\qquad \mathbf{x}\in T_i^n,
\end{equation}
computed from a stencil of neighbouring cell averages $\mathbf{q}^n_i$, and by computing volume and surface integrals using accurate quadrature formulas.
In this work, we employ only linear polynomial-based reconstructions (in the sense of Godunov's theorem) that achieve high-order accuracy on smooth flows. 
The choice of using a polynomial-based approach significantly facilitates the enforcement of the high-order boundary corrections discussed in~\cref{section:ROD}.
When working with flows involving shocks, the linear reconstruction can be easily replaced by a polynomial-based WENO approach~\cite{dumbser2007arbitrary,dumbser2007quadrature} to capture effectively the discontinuities occurring in the simulation.

\subsubsection{Local space-time predictor on moving meshes}\label{subsection:predictor}

The ADER approach is based on a local space-time predictor $\mathbf{p}_h(\mathbf{x},t)$ to achieve high-order time accuracy on deforming meshes. This predictor is constructed inside each element $T_i(t)$ over the time slab $[t^n,t^{n+1}]$ and does not rely on data from neighbouring cells. It is worth noting that, owing to the mesh motion, the space-time element exhibits different triangular geometries at the initial and final time instants $t^n$ and $t^{n+1}$. The computation of $\mathbf{p}_h(\mathbf{x},t)$ is carried out by advancing in time the polynomial $\mathbf{w}_h(\mathbf{x},t^n)$ through a weak space-time formulation of \cref{eq:euler0}. This approach originates from the work of \cite{dumbser2008unified} and was subsequently extended to moving meshes in \cite{dumbser2013arbitrary,boscheri2013arbitrary,boscheri2014direct,boscheri2015direct,gaburro2020high,gaburro2026treatment}. Therefore, we limit ourselves to presenting the essential steps, referring the reader to the aforementioned literature for a thorough description.

The space-time discretization is built upon a modified reference coordinate system that incorporates the temporal dimension. Specifically, we denote by $\overline{\overline{\mathbf x}}=(x,y,t)$ the physical coordinates and by $\boldsymbol{\overline{\overline \xi}}=(\xi,\eta,\tau)$ the corresponding reference coordinates, with $\tau \in [0,1]$. Within each space-time element $C_i^n =T_i(t)\times [t^n,t^{n+1}]$, the predictor is expressed in terms of local space-time basis functions $\theta(\overline{\overline{\mathbf x}})$ as
\begin{equation}\label{eq:ADERbasis}
\mathbf{p}_h(\mathbf{x},t) = \sum_{k=1}^Q \mathbf{\hat p}_{k} \theta_k(\overline{\overline{\mathbf x}}),\qquad \overline{\overline{\mathbf x}}\in C_i^n , \quad \text{with} \quad Q =\prod_{\ell=1}^{d+1} (p+\ell)/\ell.
\end{equation}
The same representation is adopted for the fluxes $\mathbf{F}=(\mathbf{f},\mathbf{g})$ and the source term $\mathbf{s}$, making use of the interpolation properties of the chosen nodal basis:
$$ \mathbf{f}_h(\mathbf{x},t) = \sum_{k=1}^Q \theta_k(\overline{\overline{\mathbf x}}) \mathbf{\hat f}_{k}, \quad 
\mathbf{g}_h(\mathbf{x},t)    = \sum_{k=1}^Q \theta_k(\overline{\overline{\mathbf x}}) \mathbf{\hat g}_{k}, \quad 
\mathbf{s}_h(\mathbf{x},t)    = \sum_{k=1}^Q \theta_k(\overline{\overline{\mathbf x}}) \mathbf{\hat s}_{k}. $$
Introducing the temporal mapping $t=t^n + \tau \Delta t$, the weak formulation is recast on the reference space-time element $T_e\times[0,1]$. The associated space-time Jacobian and its inverse are given by
\begin{equation}
  J = \frac{\pa \overline{\overline{\mathbf x}} }{\pa \boldsymbol{\overline{\overline{\xi}}}} = 
    \begin{bmatrix}
    x_\xi & x_\eta & x_\tau \\ y_\xi & y_\eta & y_\tau  \\ 0 & 0 & \Delta t
    \end{bmatrix}
    , \quad 
  J^{-1} = \frac{\pa \boldsymbol{\overline{\overline{\xi}}}}{\partial \overline{\overline{\mathbf x}}}  = 
    \begin{bmatrix}
    \xi_x & \xi_y & \xi_t \\ \eta_x & \eta_y & \eta_t  \\ 0 & 0 & \frac{1}{\Delta t}
    \end{bmatrix}
    .
\end{equation}
Making use of the local reference system and the space-time Jacobian $J$, the space-time formulation can be rewritten in the reference space as
\begin{equation}\label{eq:pdereferencespacetime}
\partial_\tau \mathbf{q} + \Delta t \left( \xi_t \, \partial_\xi \mathbf{q} + \eta_t \, \partial_\eta \mathbf{q} + \xi_x \, \partial_\xi \mathbf{f} + \eta_x \, \partial_\eta \mathbf{f} + \xi_y \, \partial_\xi \mathbf{g} + \eta_y \, \partial_\eta \mathbf{g} \right) = \Delta t \, \mathbf{s}(\mathbf{q}).
\end{equation}
For notational convenience, we define the inner product
$$\langle f,g \rangle = \int_0^1 \int_{T_e} f(\xi,\eta,\tau) g(\xi,\eta,\tau) \, \mathrm{d}\xi \, \mathrm{d}\eta \, \mathrm{d}\tau.$$
Substituting $\mathbf{p}_h$, $\mathbf{f}_h$, $\mathbf{g}_h$, and $\mathbf{s}_h$ into \cref{eq:pdereferencespacetime} and integrating over the space-time reference element $T_e\times[0,1]$ leads to the compact weak form 
\begin{equation}
\mathbf{K}_\tau \mathbf{\hat p}_{k} + \Delta t \left( \mathbf{K}_t \mathbf{\hat p}_{k} + \mathbf{K}_x \mathbf{\hat f}_{k} + \mathbf{K}_y \mathbf{\hat g}_{k}  \right) = \Delta t \mathbf{M}_t \mathbf{\hat s}_{k}, \label{eqn.pdeweak}
\end{equation}
wherein
\begin{equation}
\begin{split}
[\mathbf{K}_\tau]_{lk} = \big\langle \theta_l, \partial_\tau \theta_k \big\rangle, 
\quad [\mathbf{M}]_{lk} = \langle \theta_l, \theta_k \rangle, 
\quad [\mathbf{K}_t]_{lk} = \big\langle \theta_l, \xi_t \, \partial_\xi \theta_k \big\rangle + \big\langle \theta_l, \eta_t \, \partial_\eta \theta_k \big\rangle, \\ 
\quad [\mathbf{K}_x]_{lk} = \big\langle \theta_l, \xi_x \, \partial_\xi \theta_k \big\rangle + \big\langle \theta_l, \eta_x \, \partial_\eta \theta_k \big\rangle,
\quad [\mathbf{K}_y]_{lk} = \big\langle \theta_l, \xi_y \, \partial_\xi \theta_k \big\rangle + \big\langle \theta_l, \eta_y \, \partial_\eta \theta_k \big\rangle.
\end{split}
\end{equation}
It is worth pointing out that all the matrices above are computed on the reference space-time element and stored once during a preprocessing phase. Only the inverse space-time Jacobian $J^{-1}$ needs to be evaluated at each time step to account for the mesh movement. The coefficients $\mathbf{\hat p}_{k}$ of the predictor in each element $T_i(t)$ are determined independently per cell through an iterative procedure, owing to the nonlinearity of the system arising from \cref{eqn.pdeweak}. Further implementation details can be found in \cite{boscheri2013arbitrary}.

In addition, the motion of the mesh vertices must be accounted for. The trajectory of each vertex is governed by the equation
\begin{equation}\label{eq:meshspeed}
\frac{\mathrm{d}\mathbf{x}}{\mathrm{d}t} = \mathbf{V}(\mathbf{x},t),
\end{equation}
where $\mathbf{V}(\mathbf{x},t)$ denotes the local mesh velocity, which may differ from the fluid velocity. 
In this work, we specifically consider moving boundaries with a prescribed velocity. This may be viewed as a first step towards fluid-structure interaction applications.
Using the same space-time nodal expansion, the mesh position and velocity within each space-time element $T_i(t)$ are expressed as 
$$\mathbf{x}=\sum_{k=1}^Q \theta_k(\overline{\overline{\mathbf x}}) \mathbf{\hat x}_{k}\quad\text{ and }\quad 
\mathbf{V}_h=\sum_{k=1}^Q \theta_k(\overline{\overline{\mathbf x}}) \mathbf{\hat V}_{k}. $$

Following the approach of \cite{dumbser2013arbitrary}, \cref{eq:meshspeed} is solved in a space-time finite element sense:
$$\mathbf{K}_\tau \mathbf{\hat x}_{k} = \Delta t \mathbf{M} \mathbf{\hat V}_{k}.$$
This procedure yields a local space-time predictor for the nodal coordinates of the mesh.

To preserve the continuity of the discrete geometry, the mesh velocity must be uniquely defined at each vertex $\nu$. This requires a nodal solver, as the velocity is provided by the local predictor within all surrounding elements. Here, we compute a unique vertex velocity by averaging the contributions from all adjacent elements, yielding $\mathbf{\bar V}^n_\nu$. Finally, the vertex position is updated as $\mathbf{X}^{n+1}_\nu = \mathbf{X}^{n}_\nu + \Delta t \mathbf{\bar V}^n_\nu$.

\subsubsection{High order Finite Volume scheme}\label{subsection:FV}

We begin by rewriting \cref{eq:euler0} in the space-time divergence form introduced in \cite{boscheri2013arbitrary}:
\begin{equation}
\overline{\overline{\nabla}} \cdot \overline{\overline{\mathbf{Q}}} = \mathbf{s}(\mathbf{q}), \quad \text{where} \quad 
\overline{\overline{\nabla}} = \left(\frac{\partial}{\partial x},\frac{\partial}{\partial y},\frac{\partial}{\partial t}\right)^T, \quad 
\overline{\overline{\mathbf{Q}}} = \left( \mathbf{F}, \mathbf{q}\right) = \left( \mathbf{f}, \mathbf{g}, \mathbf{q}\right).  
\end{equation}
Integrating this expression over the space-time control volume $C_i^n = T_i(t)\times[t^n,t^{n+1}]$ and applying the Gauss theorem yields
\begin{equation}
\int_{\partial C_i^n} \overline{\overline{\mathbf{Q}}} \cdot \overline{\overline{\mathbf{n}}} \, \mathrm{d}S = \int_{C_i^n} \mathbf{s}(\mathbf{q}) \, \mathrm{d}\mathbf{x},
\end{equation}
where $\overline{\overline{\mathbf{n}}} = ( n_x, n_y, n_t)$ denotes the outward-oriented unit normal vector on the boundary $\partial C_i^n$ of the space-time element. The boundary $\partial C_i^n$ comprises the faces $T_i^n$, $T_i^{n+1}$, and the three lateral surfaces shared with adjacent triangles. Recalling that the space-time normals for $T_i^n$ and $T_i^{n+1}$ are $\overline{\overline{\mathbf{n}}} = (0,0,-1)$ and $\overline{\overline{\mathbf{n}}} = (0,0,1)$, respectively, the ALE one-step FV scheme can be expressed as
\begin{equation}\label{eq:alefv}
|T_i^{n+1}| \mathbf{q}_i^{n+1} = |T_i^{n}| \mathbf{q}_i^{n} - \sum_{j\in\mathcal N_i}\int_{S_{ij}^n} \overline{\overline{\mathbf{G}}}_{ij}(\mathbf{p}_h^-,\mathbf{p}_h^+, \overline{\overline{\mathbf{n}}}_{ij}) \, \mathrm{d}S + \int_{C_i^n} \mathbf{s}(\mathbf{p}_h)\, \mathrm{d}\mathbf{x} ,
\end{equation}
where $\bigcup_{j\in\mathcal N_i} S^n_{ij} = \partial C_i^n\setminus (T_i^{n} \cup T_i^{n+1})$, and $\mathcal N_i$ represents the set of neighbouring elements sharing a face with triangle $T_i(t)$.

The integrand appearing in the surface integral of \cref{eq:alefv} corresponds to the ALE numerical flux across the space-time interface $S_{ij}^n$. In the present work, we employ the Osher-type ALE flux originally proposed in the Eulerian framework \cite{dumbser2011universal} and subsequently extended to the ALE setting in \cite{dumbser2013arbitrary}:
\begin{equation}\label{eq:osherale}
  \overline{\overline{\mathbf{G}}} (\mathbf{p}^-,\mathbf{p}^+,\overline{\overline{\mathbf{n}}}) = \frac12 \left( \overline{\overline{\mathbf{Q}}}(\mathbf{p}^+) + \overline{\overline{\mathbf{Q}}}(\mathbf{p}^-) \right)\cdot \overline{\overline{\mathbf{n}}}  - \frac12 \left(\int_0^1 |\mathbf{A}_{\mathbf{n}}^{\mathbf{V}}(\boldsymbol{\Psi}(s))| \,\mathrm{d}s\right) \left(\mathbf{p}^+-\mathbf{p}^-\right) ,
\end{equation}
with the straight-line segment connecting the left and right states given by $\boldsymbol{\Psi}(s) = \mathbf{p}^- + s \left(\mathbf{p}^+-\mathbf{p}^-\right)$, $s\in[0,1]$. The matrix $\mathbf{A}_{\mathbf{n}}^{\mathbf{V}}$ denotes the ALE Jacobian projected in the spatial normal direction, defined as
\begin{equation}
  \mathbf{A}_{\mathbf{n}}^{\mathbf{V}}(\mathbf{q}) = \frac{\partial(\mathbf{F}\cdot\mathbf{n})}{\partial\mathbf{q}} - (\mathbf{V}\cdot\mathbf{n})\mathbf{I}, 
  \quad \text{where}\quad \mathbf{n} = \frac{( n_x, n_y)^T}{\sqrt{ n_x^2+ n_y^2}}.
\end{equation}
In \cref{eq:osherale}, the absolute value of the matrix is understood in the standard sense:
$$|\mathbf{A}| = \mathbf{R}|\boldsymbol{\Lambda}|\mathbf{R}^{-1}, \quad\text{where}\quad |\boldsymbol{\Lambda}| = \text{diag}(|\lambda_1|,\ldots,|\lambda_{d+2}|),$$
with $\mathbf{R}$ being the right eigenvector matrix. We note that the path integral in \cref{eq:osherale} is evaluated numerically using Gaussian quadrature.

\subsubsection{Mesh motion}\label{subsection:meshmotion}

In the present work, the computational domain is assumed to move according to a boundary velocity that is prescribed a priori. This is a standard assumption frequently adopted in the context of fluid-structure interaction applications. To consistently extend the motion from the Lagrangian boundary to the interior fluid mesh, a variety of strategies are available in the literature \cite{lohner1996improved,helenbrook2003mesh,dwight2009robust,stein2003mesh}. Among the simplest and most widely used techniques is the harmonic equation, which amounts to solving the following elliptic problem for the mesh velocity:
\begin{equation}\label{eq:laplace}
  \begin{cases}
  -\Delta \mathbf{V}(\mathbf{x},t) = 0, \quad &\text{in}\quad \Omega, \quad \\
  \mathbf{V}(\mathbf{x},t) = \mathbf{V}^{b}(\mathbf{x},t), \quad &\text{on}\quad \partial\Omega,
  \end{cases}
\end{equation}
where $\mathbf{V}^{b}(\mathbf{x},t)$ denotes the prescribed velocity imposed on the boundary $\partial\Omega$. In this work, \cref{eq:laplace} is discretised using a standard $\mathbb{P}_1$ finite element method on simplicial elements.

\subsection{Boundary conditions in the RK-DG and ADER-ALE FV frameworks}

The imposition of boundary conditions requires a framework-dependent treatment, although both approaches share the common idea of introducing a suitable ghost state.

In the RK-DG setting, whenever the element boundary $\partial T_i$ lies on the computational boundary $\partial\Omega_h$, the semi-discrete numerical flux is evaluated by defining a ghost state $\mathbf{v}_h$ and computing the boundary flux as $\hat F(\mathbf{q}_h^-,\mathbf{v}_h,\mathbf{n})$.

For space-time schemes, such as the ADER-ALE FV framework, the integration domain is the space-time element $T_i(t)\times[t^n,t^{n+1}]$. If $\partial T_i^n$ belongs to $\partial\Omega^n$, the space-time flux $\overline{\overline{\mathbf{G}}}_{ij}$ must be consistent with the boundary conditions. In this case, the neighbouring reconstructed solution is replaced by the ghost state $\mathbf{v}_h$, and the numerical flux on the space-time surfaces $S_{ij}^n$ is given by $\overline{\overline{\mathbf{G}}}_{ij}(\mathbf{p}_h^-,\mathbf{v}_h,\overline{\overline{\mathbf{n}}})$.

In this work, we focus on two types of boundary conditions.
\begin{enumerate}
  \item Weak enforcement of Dirichlet boundary conditions is achieved through the numerical fluxes, where the full ghost state $\mathbf{v}_h$ is assigned the given boundary values.
  \item The slip-wall boundary condition with non-zero speed $\mathbf{W}$ can be enforced by matching the normal component of the velocity vector
\begin{equation}\label{eq:slipwall}
\mathbf{U}\cdot\mathbf{n}=\mathbf{W}\cdot\mathbf{n} \quad\text{ on }\quad \pa\Om.
\end{equation}
In order to impose the consistent boundary condition, a high-order approximation of $\mbox{$(\mathbf{U}\cdot\mathbf{n})^\star|_{\pa\Om_h}$}$ must be computed, as explained in \cref{section:ROD}. The slip-wall condition can be reinterpreted as a Dirichlet condition on the normal velocity component.
The remaining thermodynamic variables and the tangential component of the velocity are usually taken equal to those of the internal reconstructed state, denoted with the superscript $^-$. Therefore the ghost state must yield a velocity field of the form $\mbox{$(\mathbf{U}\cdot\mathbf{n})^\star \mathbf{n} + (\mathbf{U}^-\cdot\mathbf{t}) \mathbf{t}$}$, where $\mathbf{t}$ is the tangent vector to the physical boundary. 
\end{enumerate}

\section{High-order boundary conditions on fixed and moving linear meshes}\label{section:ROD}

In this work, the goal is to avoid curvilinear mesh generation and use standard conformal linear meshes to simulate compressible flow problems on fixed and moving domains. For both DG and FV schemes, the geometry discretization of the physical domain with piecewise affine approximation introduces a second-order geometrical error that spoils the global high-order DG or FV approximation. In the following, we introduce the notation used for the physical and computational boundaries. As mentioned above, $\Om_h$ is the linear conformal mesh discretizing the physical domain $\Om$. Therefore, $\pa\Om_h$ is the polygonal approximation of $\pa\Om$. Each point on the computational boundary $\pa\Om_h$ is mapped onto the physical boundary $\pa\Om$ in a unique way:
\begin{equation} 
\forall \tilde{\mathbf{x}}\in\pa\Om_h,\quad \exists \bar{\mathbf{x}}\in\pa\Om \quad \text{s.t.}\quad \bar{\mathbf{x}}=\mathcal{M}(\tilde{\mathbf{x}}),
\end{equation} 
where $\mathcal{M}$ represents the mapping between the two boundaries. Several approaches can be used to define such mappings. In this work, we define a mapping based on a signed distance function built using the exact boundary normals $\mathbf{n}$ of the physical domain: 
$$  \bar{\mathbf{x}} =  \tilde{\mathbf{x}} +  d(\tilde{\mathbf{x}})\mathbf{n}, \quad \text{where}\quad d(\tilde{\mathbf{x}})=\|\mathbf{d}(\tilde{\mathbf{x}})\|=\|\bar{\mathbf{x}} - \tilde{\mathbf{x}}\|. $$
See \cref{fig:mapping2D} for a visual representation. 

\begin{figure}
\centering
\subfigure{
\resizebox{0.35\textwidth}{!}{%
\begin{tikzpicture}[
    >={Stealth[length=2.5mm, width=2mm]},
    punto/.style={fill, circle, inner sep=1.5pt},
    freccia/.style={->, thick, blue!70!black}
]
    % Quadrato (dominio)
    \draw[thick] (-3,-3) rectangle (3,3);
    \node[red] at (-1.85,-2.1) {$\Omega$: physical};
    \node[green] at (-1.2,-2.6) {$\Omega_h$: computational};
    
    % Cerchio al centro
    \draw[red, thick] (0,0) circle (2);

    % zoom
    \draw[black, dashed] (1.1,1.4) circle (1);
    \draw[->] (2.2, 1.6) to (4.5,1.8);

    % Decagono regolare (10 lati) che approssima il cerchio
    \def\N{7}  % numero di lati
    \def\R{2}   % raggio

    \draw[green, ultra thick] 
        ({ \R*cos(360*1/\N - 90 + 90/\N) }, 
        { \R*sin(360*1/\N - 90 + 90/\N) })
        \foreach \i in {2,...,\N} {
            -- ({ \R*cos(360*\i/\N - 90 + 90/\N) }, 
                { \R*sin(360*\i/\N - 90 + 90/\N) })
        } -- cycle;
        
\end{tikzpicture}}
}
\subfigure{
\resizebox{0.3\textwidth}{!}{%
\begin{tikzpicture}[
    >={Stealth[length=2.5mm, width=2mm]},
    punto/.style={fill, circle, inner sep=1.5pt},
    freccia/.style={->, thick, blue!70!black}
]

% ===== DOMINIO DESTRO: triangle =====
\draw[thick] (10,-3) -- (10,3) -- (7,0) -- cycle;
\draw[ultra thick,green] (10,-3) -- (10,3);

% Punti sul lato sinistro del quadrato (\tilde{x}_1, \tilde{x}_2, \tilde{x}_3)
% Lato sinistro: x=5, y da -2 a 2
\filldraw[black] (10, -1.5) circle (1.5pt) node[left] {$\tilde{\mathbf{x}}_1$};
\filldraw[black] (10, 0) circle (1.5pt) node[left] {$\tilde{\mathbf{x}}_2$};
\filldraw[black] (10, 1.5) circle (1.5pt) node[left] {$\tilde{\mathbf{x}}_3$};

% ===== DOMINIO SINISTRO: Cerchio/Semicerchio =====
% Arco di cerchio (quartiere di cerchio)
\draw[ thick,red] (10,-3) arc[start angle=-30, end angle=30, radius=6];

\draw[freccia] (10, -1.5) to (10.53,-1.8);
\draw[freccia] (10, 0) to (10.8, 0);
\draw[freccia] (10, 1.5) to (10.53,1.8);

% zoom
\draw[black, dashed] (9,0) circle (3.5);

% Punti sull'arco mappati (\bar{x}_1, \bar{x}_2, \bar{x}_3)
\filldraw[red] (10.53,-1.8) circle (2pt) node[below right] {$\bar{\mathbf{x}}_1$};
\filldraw[red] (10.8,0) circle (2pt) node[above right] {$\bar{\mathbf{x}}_2$};
\filldraw[red] (10.53,1.8) circle (2pt) node[above right] {$\bar{\mathbf{x}}_3$};

\end{tikzpicture}}
}
\caption{The mapping $\mathcal{M}$ (in blue) between the computational boundary (in green) and the physical curved boundary (in red) for 2D linear meshes.}\label{fig:mapping2D}
\end{figure}
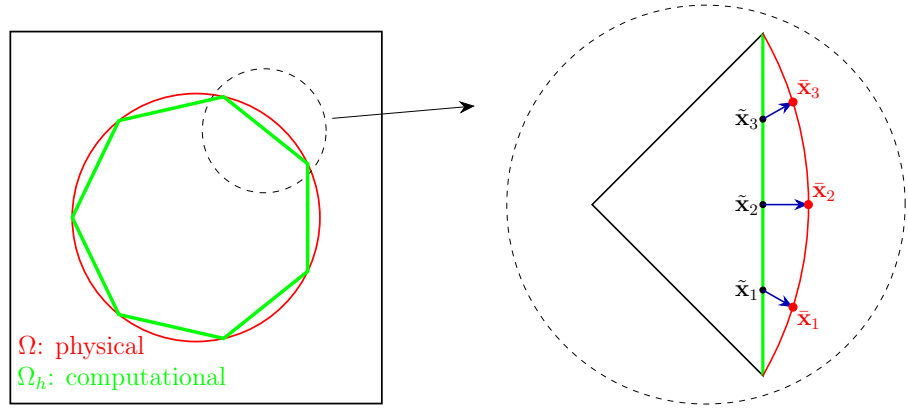

\subsection{Literature of the Reconstruction for Off-site Data method}

The Reconstruction for Off-site Data (ROD) method is a high-order boundary treatment based on a constrained minimization problem. It was firstly introduced in the context of FV schemes to preserve the high-order accuracy when discretizing curved domains with polygonal meshes \cite{costa2018very}. 
It has been used to solve a variety of physical models \cite{costa2019very,costa2021efficient,costa2022very,costa2023imposing}. 
In more recent works, it was developed in the context of DG and ADER-DG methods \cite{santos2024very,ciallella2024very}. 
In a nutshell, the ROD method constructs a modified polynomial $v_h$ on the boundary cell that exactly satisfies the boundary conditions. This polynomial is obtained by minimizing a suitable distance with respect to either the local high-order polynomial within the cell (in the DG framework) or the high-order reconstruction from a sufficiently large stencil of cell averages (in the FV framework).

Following the aforementioned notation, we refer to $\mathbf{q}_h$, when we discuss the DG approximation in space \eqref{eq:DGbasis}, and $\mathbf{w}_h$, when we discuss about the FV reconstruction \eqref{eq:FVbasis}. As introduced above, $\mathbf{p}_h$ refers to the space-time predictor in the context of ADER methods \eqref{eq:ADERbasis}, which is defined in the same way for DG or FV schemes.
In the original ROD for FV methods, the constrained minimization problem was defined to build a special reconstructed polynomial at the boundary that is as close as possible to $\mathbf{w}_h$ and is aware of the boundary conditions imposed on the physical geometry \cite{costa2018very}. In practice, the ROD method for FV becomes a least-square approach with constraints. 
To impose exactly the given boundary condition, Lagrange multipliers enforcing the condition are introduced in the functional to be minimized. The final ROD polynomial is obtained from the solution of a linear system including the standard FV reconstruction based on a stencil of cell averages, and the boundary constraints. 
When using a DG discretization in space, the ROD method simplifies significantly with respect to the FV version.
Indeed, thanks to the locality of the DG polynomial $\mathbf{q}_h$ in space, it is no longer necessary to use large stencils; the modified ROD polynomial can be directly computed by minimizing its coefficients with the ones of the internal polynomial of the boundary cell, and imposing exactly the boundary conditions through Lagrange multipliers.
This allows to reduce the main block of the linear system to the identity matrix \cite{santos2024very}, making it easier to invert. 
In the context of space-time ADER-DG method, a similar formulation in terms of the space-time predictor $\mathbf{p}_h$ was developed to obtain a single ROD polynomial in the whole space-time element that is aware of the boundary conditions on the space-time boundary surfaces \cite{ciallella2024very}.

\subsection{The standard ROD method for DG in multiple dimensions}

In this section, for simplicity, we briefly recall the ROD method originated for high-order DG schemes. 
Let us consider that the ROD method is applied on a generic scalar variable $u_h$.   
The main idea is to build a modified polynomial $v_h$ that is as close as possible to the internal polynomial $u_h$, while satisfying the boundary conditions on the real boundary. 
We consider here to build a modified polynomial $v_h$ that uses the same polynomial basis as $u_h$ and has the same number of degrees of freedom $D$:
\begin{equation}
u_h(\mathbf{x}) = \sum_{k=1}^D \psi_k(\mathbf{x}) \hat{\mathbf{u}}_k =  \boldsymbol{\varphi}^T(\mathbf{x}) \hat{\mathbf{u}}, \qquad 
v_h(\mathbf{x}) = \sum_{k=1}^D \psi_k(\mathbf{x}) \hat{\mathbf{v}}_k =  \boldsymbol{\varphi}^T(\mathbf{x}) \hat{\mathbf{v}}, 
\end{equation}
where $\boldsymbol{\varphi}=[\psi_1,\ldots,\psi_D]^T$. 

% \begin{remark}[ROD for FV]
% In the original ROD method for FV \cite{costa2018very}, a single FV reconstruction was performed to build the ROD polynomial, thus obtaining a least-square problem with constraints.
% However, the minimization problem discussed in this section, although it was introduced for DG, can be applied as a black box to FV methods as well, where the internal high-order polynomial $u_h$ used in the minimization problem is replaced by the FV reconstruction \eqref{eq:FVbasis}, which was built in advance without considering the boundary conditions.
% In this case, the FV reconstruction aware of the boundary conditions is built in two steps. 
% \end{remark}

Following the same notation, the coefficients $\hat{\mathbf{v}}$ are found by minimizing the distance to the internal polynomial coefficients $\hat{\mathbf{u}}$ subject to a set of Dirichlet boundary constraints $v_h(\bar{\mathbf{x}}_k)=u_{\D}(\bar{\mathbf{x}}_k)$, with $k=1,\ldots,K$. Let us take $\mathbf{u}_{\D}=[u_{\D}(\bar{\mathbf{x}}_1),\ldots,u_{\D}(\bar{\mathbf{x}}_K)]^T$, where $K$ is the number of boundary constraints in each boundary cell. In general, to preserve accuracy, the number of constraints is taken equal to the number of quadrature points used to compute the boundary integral. The ROD method consists in solving a minimization problem for the following functional:
\begin{equation}
\mathcal{J}(\hat{\mathbf{v}}) = \frac{1}{2} \|\hat{\mathbf{v}} - \hat{\mathbf{u}}\|_2^2 + \sum_{k=1}^K \lambda_k (v_h(\bar{\mathbf{x}}_k) - u_{\D}(\bar{\mathbf{x}}_k)),
\end{equation}
where the Euclidean norm is used to measure the distance between the two polynomials, and $\boldsymbol{\lambda}=[\lambda_1,\ldots,\lambda_K]^T$ is the vector of Lagrange multipliers. The optimality conditions for this problem give:
\begin{equation}
  \frac{\partial \mathcal{J}}{\partial \hat{\mathbf{v}}} = \hat{\mathbf{v}} - \hat{\mathbf{u}} + \Phi(\bar{\mathbf{x}}) \boldsymbol{\lambda} = 0, \qquad
  \frac{\partial \mathcal{J}}{\partial \boldsymbol{\lambda}} = \Phi^T(\bar{\mathbf{x}}) \hat{\mathbf{v}} - \mathbf{u}_{\D} = 0, 
\end{equation}
where $[\Phi(\bar{\mathbf{x}})]_{jk}=\varphi_j(\bar{\mathbf{x}}_k)$ is of full column rank.
Thus, the following linear system is solved to obtain the coefficients $\hat{\mathbf{v}}$ of the modified polynomial:
\begin{equation}\label{eq:RODsystem}
\begin{bmatrix}\mathbf{I} & \Phi(\bar{\mathbf{x}}) \\
\Phi^T(\bar{\mathbf{x}}) & \mathbf{0} \end{bmatrix}
\begin{bmatrix}\hat{\mathbf{v}} \\ \boldsymbol{\lambda} \end{bmatrix} =
\begin{bmatrix}\hat{\mathbf{u}} \\ \mathbf{u}_{\D} \end{bmatrix}.
\end{equation}  
Then the coefficients $\hat{\mathbf{v}}$ are employed to compute the modified boundary conditions in the corresponding quadrature points on the discretized boundary $\{\tilde{\mathbf{x}}_k\}_{k=1}^K$ (see \cref{fig:mapping2D} for a visual representation), with the scalar product $v_h(\tilde{\mathbf{x}}_k) = \boldsymbol{\varphi}^T(\tilde{\mathbf{x}}_k) \hat{\mathbf{v}}$. To impose the high-order boundary conditions on the computational boundary, the corrected value $v_h(\tilde{\mathbf{x}}_k)$ is computed for each conservative variable, and is then used as a ghost state in the boundary numerical flux.

The system \eqref{eq:RODsystem} can be solved by using the Schur complement, which reduces the problem to solving
\begin{equation}
\hat{\mathbf{v}} = \hat{\mathbf{u}} - \Phi(\bar{\mathbf{x}}) \boldsymbol{\lambda}, \qquad
\boldsymbol{\lambda} = (\Phi^T(\bar{\mathbf{x}}) \Phi(\bar{\mathbf{x}}))^{-1} (\Phi^T(\bar{\mathbf{x}}) \hat{\mathbf{u}} - \mathbf{u}_{\D}).
\end{equation}
To solve this linear system, we need to invert a matrix of size $K\times K$. Thus, the computational complexity of this approach is $\mathcal{O}(K^3+K^2D + KD)$.

\begin{remark}[Extension to moving meshes]
Although, for fixed meshes, the inversion of the matrix $\Phi^T(\bar{\mathbf{x}}) \Phi(\bar{\mathbf{x}})$ can be precomputed and stored, for moving meshes, the matrix changes at each time step, and thus it needs to be inverted at each time step for each boundary cell.  
\end{remark}

With the goal of reducing significantly the computational costs related to this high-order boundary correction, in the following section, we show how to write the ROD method as a simple polynomial correction that no longer needs the assembly and inversion of $\Phi^T(\bar{\mathbf{x}}) \Phi(\bar{\mathbf{x}})$. Moreover, we will show that the same concept can be applied as a black box in several frameworks (DG, FV, space-time one-step ADER).

\subsection{ROD method as a 1D polynomial correction}\label{subsection:ROD1Dcorrection}

In this work, we propose a new interpretation of the ROD method as a 1D-like polynomial correction along the mapping direction that links a point on the computational boundary $\tilde{\mathbf{x}}$ to the one on the physical boundary $\bar{\mathbf{x}}$. This approach allows us to extend the ROD method to moving meshes without the need to invert the aforementioned $K\times K$ matrix at each time step, thus significantly reducing the computational cost. 
The problem arises from the use of multiple boundary constraints in the minimization problem required to preserve consistency \cite{santos2024very}, which makes the inversion of $\Phi^T(\bar{\mathbf{x}}) \Phi(\bar{\mathbf{x}})$ unavoidable.

Therefore, the main idea is to build the modified polynomial $v_h$ by minimizing the distance to the internal polynomial $u_h$ along the mapping direction to the boundary, while satisfying the corresponding boundary condition on the real boundary. In this way, the minimization problem only involves one constraint for each quadrature point on the boundary, and thus it can be solved analytically without the need to invert any matrix. We also show that this approach can be recast as an elegant polynomial correction, in the spirit of the shifted boundary polynomial correction introduced in \cite{ciallella2023shifted,ciallella2025stability}, to simplify the development of the high-order shifted boundary method \cite{Scovazzi1}. This new formulation highly simplifies the development of minimization-based boundary corrections by making them straightforward to be introduced in several computational frameworks. 

In particular, following the one-dimensional analysis carried out in \cite{ciallella2025minimization}, we show here that the proof to achieve this formulation can be extended straightforwardly in the context of FV, DG and space-time ADER methods with fixed and moving multi-dimensional domains, and obtain a general methodology. 

In particular, we focus on two different minimization-based approaches which rely on different convex distance functions: the Euclidean norm and the $L^2$-norm. 
We introduce two new methods, denoted as ROD-E and ROD-$L^2$, respectively. Unlike the original ROD method, which computes a single modified polynomial per boundary cell, these methods rely on a different minimization formulation that yields a modified polynomial at each quadrature point. 
In one spatial dimension, ROD-E is equivalent to a reformulation of the original ROD method. In multiple dimensions, however, the new formulations differ from the original one, which needed the inversion of the constraint matrix.

\subsubsection{Euclidean norm as distance measure}

Let us consider to write the ROD linear system \eqref{eq:RODsystem} with only one boundary constraint. Rather than building one ROD polynomial for the whole boundary element, we propose to build as many ROD polynomials as the points $\tilde{\mathbf{x}}$ on the computational boundary. For each one of the ROD polynomials, we enforce only the boundary condition at the corresponding point $\bar{\mathbf{x}}$ on the physical boundary (see \cref{fig:mapping2D}).

When using the Euclidean distance, the cost functional for the ROD-E method is given by
\begin{equation}
\mathcal{J}(\hat{\mathbf{v}},\lambda) = \frac{1}{2} \|\hat{\mathbf{v}} - \hat{\mathbf{u}}\|^2_2 + \lambda (v_h(\bar{\mathbf{x}}) - u_{\D}(\bar{\mathbf{x}})).
\end{equation}
The optimality conditions for this problem can be written as
\begin{equation}
\frac{\partial \mathcal{J}}{\partial \hat{\mathbf{v}}} = \hat{\mathbf{v}} - \hat{\mathbf{u}} + \lambda \boldsymbol{\varphi}(\bar{\mathbf{x}}) = 0, \qquad
\frac{\partial \mathcal{J}}{\partial \lambda} = \boldsymbol{\varphi}^T(\bar{\mathbf{x}})\hat{\mathbf{v}} - u_{\D}(\bar{\mathbf{x}}) = 0.
\end{equation}
The linear system in this case is simply given by
\begin{equation}\begin{bmatrix}\mathbf{I} & \boldsymbol{\varphi}(\bar{\mathbf{x}}) \\
\boldsymbol{\varphi}^T(\bar{\mathbf{x}}) & 0 \end{bmatrix}
\begin{bmatrix}\hat{\mathbf{v}} \\ \lambda \end{bmatrix} =
\begin{bmatrix}\hat{\mathbf{u}} \\ u_{\D}(\bar{\mathbf{x}}) \end{bmatrix}.
\end{equation}
Once the coefficients $\hat{\mathbf{v}}$ of the modified polynomial are obtained, the boundary conditions at the point $\tilde{\mathbf{x}}$ on the discretized boundary can be computed as $v_h(\tilde{\mathbf{x}}) = \boldsymbol{\varphi}^T(\tilde{\mathbf{x}}) \hat{\mathbf{v}}$.

\begin{proposition}
  Imposing boundary conditions with the single-constraint ROD-E method on the discretized boundary $\tilde{\mathbf{x}}$ is equivalent to the following polynomial correction:
\begin{equation}\label{eq:RODE}
v_h(\tilde{\mathbf{x}}) = [\boldsymbol{\varphi}(\tilde{\mathbf{x}}) - \alpha^{\text{ROD-E}} \boldsymbol{\varphi}(\bar{\mathbf{x}})]^T \hat{\mathbf{u}} + \alpha^{\text{ROD-E}} u_{\D}(\bar{\mathbf{x}}),
\end{equation}
where 
\begin{equation}
\alpha^{\text{ROD-E}} = \frac{\boldsymbol{\varphi}^T(\tilde{\mathbf{x}}) \boldsymbol{\varphi}^T(\bar{\mathbf{x}})}{\boldsymbol{\varphi}^T(\bar{\mathbf{x}}) \boldsymbol{\varphi}(\bar{\mathbf{x}})} = \frac{\boldsymbol{\varphi}^T(\tilde{\mathbf{x}}) \boldsymbol{\varphi}(\bar{\mathbf{x}})}{\|\boldsymbol{\varphi}(\bar{\mathbf{x}})\|_2^2}.
\end{equation}
\end{proposition}

\begin{proof}
The polynomial correction can be retrieved by solving the minimization problem. From the first optimality condition, we have 
$$ \hat{\mathbf{v}} = \hat{\mathbf{u}} - \lambda \boldsymbol{\varphi}(\bar{\mathbf{x}}) .$$
By replacing this expression in the second optimality condition, we get
$$ \lambda = \frac{\boldsymbol{\varphi}^T(\bar{\mathbf{x}}) \hat{\mathbf{u}} - u_{\D}(\bar{\mathbf{x}})}{\boldsymbol{\varphi}^T(\bar{\mathbf{x}}) \boldsymbol{\varphi}(\bar{\mathbf{x}})} .$$
Finally, the modified degrees of freedom $\hat{\mathbf{v}}$ can be written as
$$ \hat{\mathbf{v}} = \hat{\mathbf{u}} - \frac{\boldsymbol{\varphi}^T(\bar{\mathbf{x}}) \hat{\mathbf{u}} - u_{\D}(\bar{\mathbf{x}})}{\boldsymbol{\varphi}^T(\bar{\mathbf{x}}) \boldsymbol{\varphi}(\bar{\mathbf{x}})} \boldsymbol{\varphi}(\bar{\mathbf{x}}) .$$
The polynomial correction can be obtained by evaluating the modified polynomial in the corresponding point on the computational boundary $\tilde{\mathbf{x}}$, which gives
\begin{align*}
v_h(\tilde{\mathbf{x}}) &= \boldsymbol{\varphi}^T(\tilde{\mathbf{x}}) \hat{\mathbf{v}} = \boldsymbol{\varphi}^T(\tilde{\mathbf{x}}) \hat{\mathbf{u}} - \frac{\boldsymbol{\varphi}^T(\bar{\mathbf{x}}) \hat{\mathbf{u}} - u_{\D}(\bar{\mathbf{x}})}{\boldsymbol{\varphi}^T(\bar{\mathbf{x}}) \boldsymbol{\varphi}(\bar{\mathbf{x}})} \boldsymbol{\varphi}^T(\tilde{\mathbf{x}}) \boldsymbol{\varphi}(\bar{\mathbf{x}}) \\
&= \boldsymbol{\varphi}^T(\tilde{\mathbf{x}}) \hat{\mathbf{u}} - \frac{\boldsymbol{\varphi}^T(\bar{\mathbf{x}}) \hat{\mathbf{u}}}{\boldsymbol{\varphi}^T(\bar{\mathbf{x}}) \boldsymbol{\varphi}(\bar{\mathbf{x}})} \boldsymbol{\varphi}^T(\tilde{\mathbf{x}}) \boldsymbol{\varphi}(\bar{\mathbf{x}}) + \frac{u_{\D}(\bar{\mathbf{x}})}{\boldsymbol{\varphi}^T(\bar{\mathbf{x}}) \boldsymbol{\varphi}(\bar{\mathbf{x}})} \boldsymbol{\varphi}^T(\tilde{\mathbf{x}}) \boldsymbol{\varphi}(\bar{\mathbf{x}})  \\
&= \left[\boldsymbol{\varphi}(\tilde{\mathbf{x}}) - \frac{\boldsymbol{\varphi}^T(\tilde{\mathbf{x}}) \boldsymbol{\varphi}(\bar{\mathbf{x}})}{\boldsymbol{\varphi}^T(\bar{\mathbf{x}}) \boldsymbol{\varphi}(\bar{\mathbf{x}})} \boldsymbol{\varphi}(\bar{\mathbf{x}})\right]^T \hat{\mathbf{u}} + \frac{\boldsymbol{\varphi}^T(\tilde{\mathbf{x}}) \boldsymbol{\varphi}(\bar{\mathbf{x}})}{\boldsymbol{\varphi}^T(\bar{\mathbf{x}}) \boldsymbol{\varphi}(\bar{\mathbf{x}})} u_{\D}(\bar{\mathbf{x}}),
\end{align*}
which concludes the proof.
\end{proof}

The computational complexity of this approach is drastically reduced compared to the original ROD method.
Indeed, the ROD-E polynomial correction has computational complexity of $\mathcal{O}(KD)$, which is significantly better than the $\mathcal{O}(K^3+K^2D + KD)$ of the original ROD method. The polynomial correction for each point $\mathbf{\tilde{x}}$ has a computational cost of $\mathcal{O}(D)$, given by the scalar products used to compute $\alpha^{\text{ROD-E}}$ and evaluate the modified polynomial in $\tilde{\mathbf{x}}$. Since there are $K$ quadrature points on the boundary, the total complexity is $\mathcal{O}(KD)$ for each boundary cell.

\subsubsection{$L^2$-norm as distance measure}

For the ROD-$L^2$ method, the cost functional to minimize is given by
\begin{equation}
\mathcal{J}(\hat{\mathbf{v}},\lambda) = \frac{1}{2} \int_{T} (v_h(\mathbf{x}) - u_h(\mathbf{x}))^2 \diff{\mathbf{x}} + \lambda (v_h(\bar{\mathbf{x}}) - u_{\D}(\bar{\mathbf{x}})).
\end{equation}
The optimality conditions for this problem can be written as
\begin{equation}
\frac{\partial \mathcal{J}}{\partial \hat{\mathbf{v}}} = \int_{T} (v_h(\mathbf{x}) - u_h(\mathbf{x})) \boldsymbol{\varphi}(\mathbf{x}) \diff{\mathbf{x}} + \lambda \boldsymbol{\varphi}(\bar{\mathbf{x}}) = 0, \qquad
\frac{\partial \mathcal{J}}{\partial \lambda} = \boldsymbol{\varphi}^T(\bar{\mathbf{x}})\hat{\mathbf{v}} - u_{\D}(\bar{\mathbf{x}}) = 0.
\end{equation}
The linear system in this case is simply given by
\begin{equation}\begin{bmatrix}M & \boldsymbol{\varphi}(\bar{\mathbf{x}}) \\
\boldsymbol{\varphi}^T(\bar{\mathbf{x}}) & 0 \end{bmatrix}
\begin{bmatrix}\hat{\mathbf{v}} \\ \lambda \end{bmatrix} = \begin{bmatrix}M \hat{\mathbf{u}} \\ u_{\D}(\bar{\mathbf{x}}) \end{bmatrix},
\end{equation}
where $M$ is the mass matrix. Once the coefficients $\hat{\mathbf{v}}$ of the modified polynomial are obtained, the boundary conditions in the quadrature points on the discretized boundary can be computed as $v_h(\tilde{\mathbf{x}}) = \boldsymbol{\varphi}^T(\tilde{\mathbf{x}}) \hat{\mathbf{v}}$. 

\begin{proposition}
  Imposing boundary conditions with the single-constraint ROD-$L^2$ method on the discretized boundary $\tilde{\mathbf{x}}$ is equivalent to the following polynomial correction:
\begin{equation}\label{eq:RODL2}
v_h(\tilde{\mathbf{x}}) = [\boldsymbol{\varphi}(\tilde{\mathbf{x}}) - \alpha^{\text{ROD-}L^2} \boldsymbol{\varphi}(\bar{\mathbf{x}})]^T \hat{\mathbf{u}} + \alpha^{\text{ROD-}L^2} u_{\D}(\bar{\mathbf{x}}),
\end{equation}
where 
\begin{equation}
\alpha^{\text{ROD-}L^2} = \frac{\boldsymbol{\varphi}^T(\tilde{\mathbf{x}}) M^{-1} \boldsymbol{\varphi}(\bar{\mathbf{x}})}{\boldsymbol{\varphi}^T(\bar{\mathbf{x}}) M^{-1} \boldsymbol{\varphi}(\bar{\mathbf{x}})} = \frac{\boldsymbol{\varphi}^T(\tilde{\mathbf{x}}) M^{-1} \boldsymbol{\varphi}(\bar{\mathbf{x}})}{ \|\boldsymbol{\varphi}(\bar{\mathbf{x}})\|_{M^{-1}}} .
\end{equation}
\end{proposition}

\begin{proof}
The polynomial correction can be retrieved by solving the minimization problem. The optimality conditions give 
$$ \hat{\mathbf{v}} = \hat{\mathbf{u}} - \lambda M^{-1} \boldsymbol{\varphi}(\bar{\mathbf{x}}),\qquad \lambda = \frac{\boldsymbol{\varphi}^T(\bar{\mathbf{x}}) \hat{\mathbf{u}} - u_{\D}(\bar{\mathbf{x}})}{\boldsymbol{\varphi}^T(\bar{\mathbf{x}}) M^{-1} \boldsymbol{\varphi}(\bar{\mathbf{x}})} .$$
The polynomial correction can be obtained by evaluating the modified polynomial in the quadrature points on the discretized boundary $\tilde{\mathbf{x}}$, which gives
\begin{align*}
v_h(\tilde{\mathbf{x}}) &= \boldsymbol{\varphi}^T(\tilde{\mathbf{x}}) \hat{\mathbf{v}} = \boldsymbol{\varphi}^T(\tilde{\mathbf{x}}) \hat{\mathbf{u}} - \frac{\boldsymbol{\varphi}^T(\bar{\mathbf{x}}) \hat{\mathbf{u}} - u_{\D}(\bar{\mathbf{x}})}{\boldsymbol{\varphi}^T(\bar{\mathbf{x}}) M^{-1} \boldsymbol{\varphi}(\bar{\mathbf{x}})} \boldsymbol{\varphi}^T(\tilde{\mathbf{x}}) M^{-1} \boldsymbol{\varphi}(\bar{\mathbf{x}}) \\
&= \boldsymbol{\varphi}^T(\tilde{\mathbf{x}}) \hat{\mathbf{u}}  - \frac{\boldsymbol{\varphi}^T(\bar{\mathbf{x}}) \hat{\mathbf{u}}}{\boldsymbol{\varphi}^T(\bar{\mathbf{x}}) M^{-1} \boldsymbol{\varphi}(\bar{\mathbf{x}})} \bold{\varphi}^T(\tilde{\mathbf{x}}) M^{-1} \boldsymbol{\varphi}(\bar{\mathbf{x}}) + \frac{u_{\D}(\bar{\mathbf{x}})}{\boldsymbol{\varphi}^T(\bar{\mathbf{x}}) M^{-1} \boldsymbol{\varphi}(\bar{\mathbf{x}})} \boldsymbol{\varphi}^T(\tilde{\mathbf{x}}) M^{-1} \boldsymbol{\varphi}(\bar{\mathbf{x}})\\
&= \left[\boldsymbol{\varphi}(\tilde{\mathbf{x}}) - \frac{\boldsymbol{\varphi}(\tilde{\mathbf{x}}) M^{-1} \boldsymbol{\varphi}(\bar{\mathbf{x}})}{\boldsymbol{\varphi}^T(\bar{\mathbf{x}}) M^{-1} \boldsymbol{\varphi}(\bar{\mathbf{x}})} \boldsymbol{\varphi}^T(\bar{\mathbf{x}})\right]^T \hat{\mathbf{u}} + \frac{\boldsymbol{\varphi}^T(\tilde{\mathbf{x}}) M^{-1} \boldsymbol{\varphi}(\bar{\mathbf{x}})}{\boldsymbol{\varphi}^T(\bar{\mathbf{x}}) M^{-1} \boldsymbol{\varphi}(\bar{\mathbf{x}})} u_{\D}(\bar{\mathbf{x}}),
\end{align*}
which concludes the proof.
\end{proof}

The computational complexity of this approach is also smaller than the original ROD method, although it is higher due to the presence of the mass matrix. Let us consider that the inversion of the mass matrix does not affect the computational cost of the new boundary treatment, since it has to be computed anyway by the numerical scheme. Thus, the ROD-$L^2$ polynomial correction has a computational complexity of $\mathcal{O}(KD^2+KD)$ for each boundary cell.

Although ROD-$L^2$ is less efficient than ROD-E, we propose to test both variants to show that both methods are consistent and are able to achieve high-order accuracy for fixed and moving curved domains discretized with 2D and 3D simplicial elements. In addition to that, despite the slightly larger computational cost, ROD-$L^2$ has shown to have better stability properties than ROD-E \cite{ciallella2025minimization}, which could be relevant for very high-order computations and complex applications. 

\begin{remark}[A unified correction strategy for DG, FV and ADER methods]
Both the ROD-E and ROD-$L^2$ polynomial corrections can be straightforwardly employed in several computational frameworks, such as DG, FV, and space-time ADER methods. Provided that the internal solution in the boundary element is expressed in terms of a polynomial basis (see \cref{eq:DGbasis}, \cref{eq:FVbasis}, and \cref{eq:ADERbasis} for FV, DG, and ADER methods, respectively), the polynomial corrections \eqref{eq:RODE} and \eqref{eq:RODL2} can be applied directly by simply replacing the degrees of freedom $\hat{\mathbf{u}}$ with $\hat{\mathbf{q}}_h$, $\hat{\mathbf{w}}_h$, and $\hat{\mathbf{p}}_h$, respectively. In the context of ADER methods, the set of basis functions $\boldsymbol{\varphi}$ is replaced by the space-time basis $\boldsymbol{\theta}=[\theta_1,...,\theta_Q]^T$ used for the ADER predictor.
To show this flexibility, herein, we validate the  ROD-E and ROD-$L^2$ boundary corrections in two computational frameworks: the RK-DG on fixed domains, and the more advanced ADER-ALE FV on moving domains. 
\end{remark}

\section{Numerical experiments in the RK-DG framework}\label{section:ResultsDG}

\subsection{Manufactured solution on 2D fixed meshes}\label{subsection:ManufacturedFixed}

We start by assessing the accuracy of the novel boundary conditions on two-dimensional fixed meshes. 
We consider a smooth manufactured solution within a circular domain. This is achieved by introducing a source term into the two-dimensional inhomogeneous Euler equations:
\[
\mathbf{s} = \begin{bmatrix} 0.4\cos(x+y) \\ 0.6\cos(x+y) \\ 0.6\cos(x+y) \\ 1.8\cos(x+y) \end{bmatrix}.
\]
The exact steady-state solution is given by
\[
\rho = 1 + 0.2 \sin(x+y),\quad U_1 = 1,\quad U_2 = 1,\quad p = 1 + 0.2 \sin(x+y),
\]
which is imposed on the boundary via a ghost state. We here compare the solutions obtained with the ROD-E and ROD-$L^2$ corrections and the solution obtained without any correction, by imposing the boundary condition of the physical boundary onto the computational one.

In \cref{fig:manuf2Dfix}, we present the numerical results obtained with a DG-$\mathbb{P}_4$ scheme by comparing the solutions obtained with and without the ROD-E correction. Qualitatively, we can already appreciate the impact of the boundary correction. In the convergence analysis presented in \cref{tab:manuf2D}, we can see even more the impact of the corrections on the final numerical solution, which cannot be better than second-order accurate if no correction is applied. It can be noticed that almost no improvement is observed when increasing the polynomial degree on the same mesh.
On the contrary, both ROD-E and ROD-$L^2$ methods appear to achieve arbitrary high-order accuracy, validated up to fifth order. 

\begin{figure}
  \centering
  \subfigure[]{\includegraphics[width=0.48\textwidth]{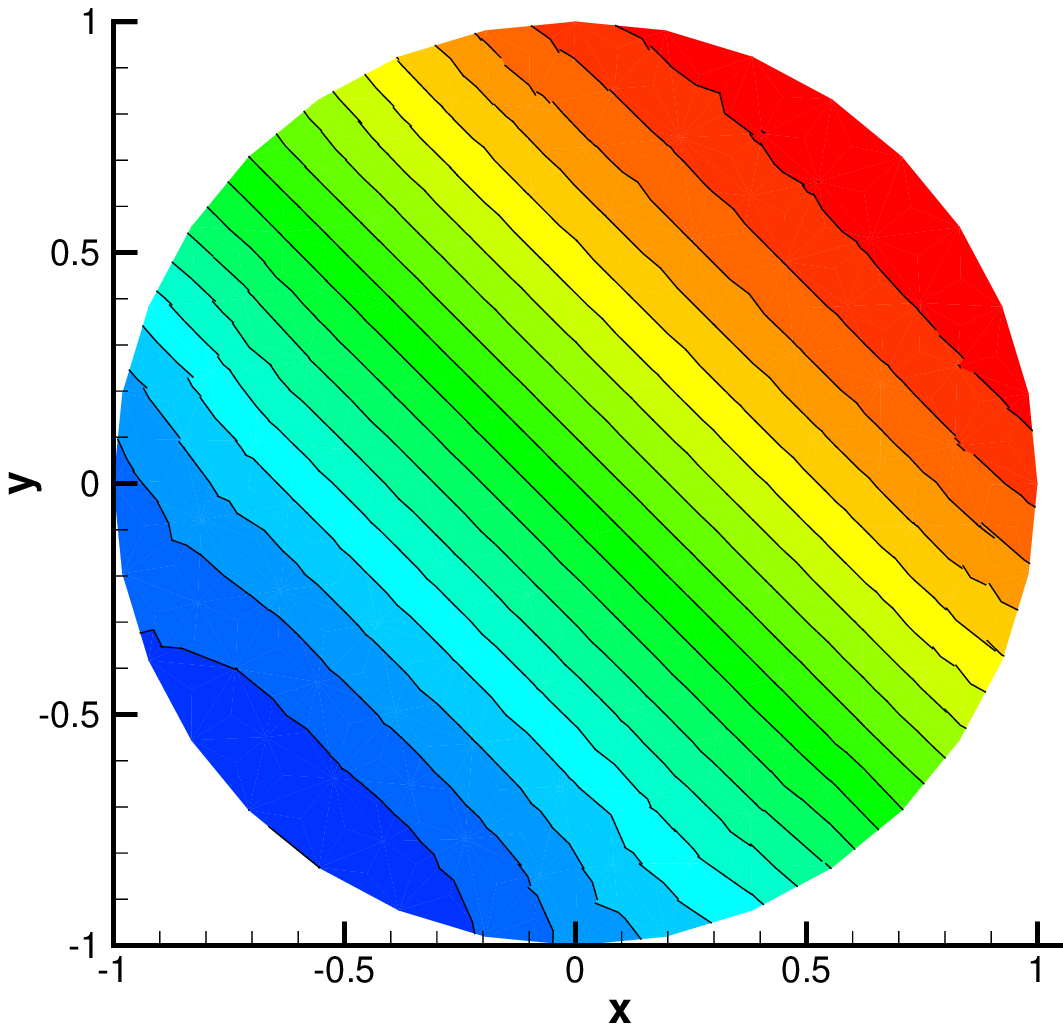}}
  \subfigure[]{\includegraphics[width=0.48\textwidth]{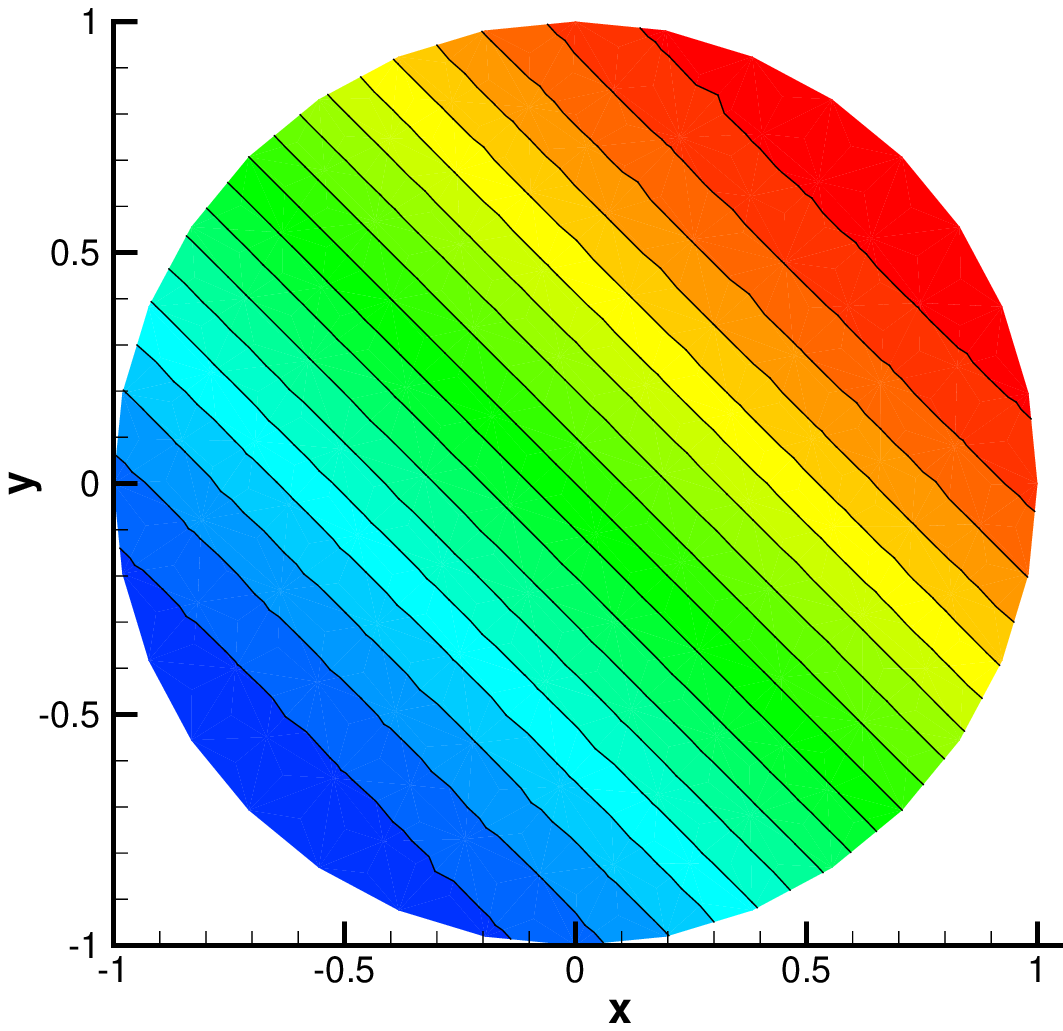}}
  \caption{Manufactured solution on 2D fixed meshes: DG-$\mathbb{P}_4$ numerical solution for the test case presented in~\cref{subsection:ManufacturedFixed}. Numerical solution obtained without (left) and with (right) the ROD-E polynomial correction.}\label{fig:manuf2Dfix}
\end{figure}

\setlength{\tabcolsep}{3pt}  % default è 6pt
\begin{table}
  \caption{Manufactured solution on 2D fixed meshes: convergence analysis at $t=0.1$ for the test case presented in~\cref{subsection:ManufacturedFixed}. Numerical results obtained with the RK-DG framework and Dirichlet-type boundary conditions imposed with and w/o the ROD-E and ROD-$L^2$ polynomial corrections on fixed conformal linear meshes.}\label{tab:manuf2D}
  \footnotesize
  \centering
  \begin{tabular}{ccccccccccccc}
          \hline\hline          &\multicolumn{2}{c}{$\rho$} &\multicolumn{2}{c}{$\rho U_1$}   &\multicolumn{2}{c}{$\rho$} &\multicolumn{2}{c}{$\rho U_1$}  &\multicolumn{2}{c}{$\rho$} &\multicolumn{2}{c}{$\rho U_1$}  \\[0.5mm]
          \cline{2-13}
          Grid size & Error        & Order & Error        & Order & Error      & Order  & Error      & Order & Error      & Order  & Error      & Order  \\[0.5mm]\hline
          &\multicolumn{12}{c}{DG-$\mathbb{P}_1$}\\
          &\multicolumn{4}{c}{w/o correction}   &\multicolumn{4}{c}{ROD-E}  &\multicolumn{4}{c}{ROD-$L^2$} \\[0.5mm]
1.05E-1   & 3.92E-04  &   --   & 4.42E-04  &   --  & 3.96E-04  &  --  & 4.29E-04  &  --  &  3.95E-04 &   -- &	4.29E-04 &   --  \\
5.27E-2   & 9.78E-05  &  2.00  & 1.11E-04  &  1.99 & 9.74E-05  & 2.02 & 1.07E-04  & 2.00 &  9.74E-05 & 2.01	& 1.07E-04 & 2.00  \\
2.63E-2   & 2.43E-05  &  2.00  & 2.77E-05  &  2.00 & 2.41E-05  & 2.01 & 2.66E-05  & 2.00 &  2.41E-05 & 2.01	& 2.66E-05 & 2.00  \\
1.32E-2   & 6.07E-06  &  2.00  & 6.91E-06  &  2.00 & 5.97E-06  & 2.01 & 6.63E-06  & 2.00 &  5.97E-06 & 2.00	& 6.62E-06 & 2.00  \\
          &\multicolumn{12}{c}{DG-$\mathbb{P}_2$}\\
          &\multicolumn{4}{c}{w/o correction}   &\multicolumn{4}{c}{ROD-E}  &\multicolumn{4}{c}{ROD-$L^2$} \\[0.5mm]
1.05E-1   & 1.20E-04  &   --   & 1.61E-04  &  --   & 1.31E-05  &  --  &  1.42E-05   &  --  & 1.31E-05  &  --  & 1.42E-05 &  --   \\
5.27E-2   & 2.89E-05  &  2.05  & 3.99E-05  &  2.01 & 1.85E-06  & 2.82 &  1.81E-06   & 2.96 & 1.85E-06  & 2.82 & 1.81E-06 &  2.96  \\
2.63E-2   & 7.14E-06  &  2.01  & 9.97E-06  &  2.00 & 2.80E-07  & 2.72 &  2.42E-07   & 2.90 & 2.80E-07  & 2.72 & 2.42E-07 &  2.90  \\
1.32E-2   & 1.77E-06  &  2.01  & 2.48E-06  &  2.00 & 4.52E-08  & 2.63 &  3.42E-08   & 2.83 & 4.51E-08  & 2.63 & 3.42E-08 &  2.82  \\
          &\multicolumn{12}{c}{DG-$\mathbb{P}_3$}\\
          &\multicolumn{4}{c}{w/o correction}   &\multicolumn{4}{c}{ROD-E}  &\multicolumn{4}{c}{ROD-$L^2$} \\[0.5mm]
1.05E-1   & 1.16E-04  &   -- & 1.61E-04   &   --  & 1.38E-07  &  --  & 1.58E-07  &   --  & 1.35E-07  & --    & 1.54E-07  &   --  \\
5.27E-2   & 2.89E-05  & 1.99 & 4.06E-05   &  1.98 & 8.48E-09  & 4.02 & 9.54E-09  &  4.05 & 8.42E-09  & 4.00  & 9.44E-09  &  4.02  \\
2.63E-2   & 7.19E-06  & 2.00 & 1.01E-05   &  2.00 & 5.22E-10  & 4.02 & 5.86E-10  &  4.02 & 5.20E-10  & 4.01  & 5.84E-10  &  4.01  \\
1.32E-2   & 1.79E-06  & 2.00 & 2.52E-06   &  2.00 & 3.20E-11  & 4.02 & 3.63E-11  &  4.01 & 3.20E-11  & 4.02  & 3.62E-11  &  4.00  \\
          &\multicolumn{12}{c}{DG-$\mathbb{P}_4$}\\
          &\multicolumn{4}{c}{w/o correction}   &\multicolumn{4}{c}{ROD-E}  &\multicolumn{4}{c}{ROD-$L^2$} \\[0.5mm]
1.05E-1   & 1.15E-04 & -- & 1.63E-04 &   -- & 2.71E-09 & 	--  & 2.88E-09 &  --  & 2.70E-09	& 	--  & 2.87E-09 &  -- \\
5.27E-2   & 2.90E-05 & 1.99 & 4.10E-05 & 1.98 & 9.65E-11 & 4.81 & 9.32E-11 & 4.94 & 9.64E-11	& 4.80  & 9.31E-11 & 4.94  \\
2.63E-2   & 7.22E-06 & 2.00 & 1.02E-05 & 2.00 & 3.66E-12 & 4.72 & 3.18E-12 & 4.87 & 3.66E-12	& 4.72  & 3.18E-12 & 4.87  \\
1.32E-2   & 1.80E-06 & 2.00 & 2.55E-06 & 2.00 & 7.03E-14 & 5.70 & 7.82E-14 & 5.34 & 7.03E-14	& 5.70  & 7.82E-14 & 5.34  \\
          \hline\hline\\[1pt]
  \end{tabular}
  \end{table}

\subsection{Supersonic vortex bounded by two circular walls on 2D fixed meshes}\label{subsection:vortex}

To assess the accuracy of the novel wall boundary conditions, we simulate an isentropic supersonic vortex flow confined between two concentric circular arcs. The inner and outer radii are given by $r_i = 1$ and $r_o = 1.384$, respectively. On the inner circle, the Mach number is $M_i = 2.25$ and the density is set to $\rho_i = 1$. The exact density distribution depends on the radial coordinate $r$ according to
\begin{equation}
\rho = \rho_i \left( 1 + \frac{\gamma+1}{2}M_i^2 \left(1- \left(\frac{r_i}{r}\right)^2 \right) \right)^{\frac{1}{\gamma-1}}.
\end{equation}
The velocity magnitude and pressure are expressed as
\begin{equation}
\|\mathbf{U}\|=\frac{c_i M_i}{r}, \qquad p = \frac{\rho^\gamma}{\gamma},
\end{equation}
where $c_i$ denotes the speed of sound at the inner radius. Finally, the Cartesian velocity components are obtained from
\begin{equation}
\begin{bmatrix} U_1 \\ U_2 \end{bmatrix} = \|\mathbf{U}\| \begin{bmatrix} y/r \\ -x/r \end{bmatrix},
\end{equation}
with $r = \sqrt{x^2 + y^2}$.

For this test, we aim at validating the ROD-E and ROD-$L^2$ boundary corrections on the more complex supersonic vortex bounded by two circular walls, which needs the consistent development of high-order slip-wall boundary conditions.
As also observed in previous works \cite{ciallella2023shifted,krivodonova2006high} and in \cref{tab:vortex2D}, by simply reverting the velocity vector at the wall as it is commonly done to impose slip-wall boundary conditions, discretization errors and convergence orders are extremely compromised in the presence of curved boundaries. In particular, when fixing a given mesh and increasing the polynomial degree, the standard slip-wall condition with no correction produces even higher discretization errors and convergence slopes lower than second-order. For $\mathbb{P}_4$, the order-of-accuracy even drops to first order with discretization errors one order of magnitude higher than those obtained with $\mathbb{P}_1$.
Both ROD-E and ROD-$L^2$ corrections recover all the high-order convergence slopes with discretization errors that tend towards machine precision. The impact of the boundary corrections is much visible in this case and can be visualized in \cref{fig:vortex2D}. 

\begin{figure}
  \centering
  \subfigure[]{\includegraphics[width=0.48\textwidth]{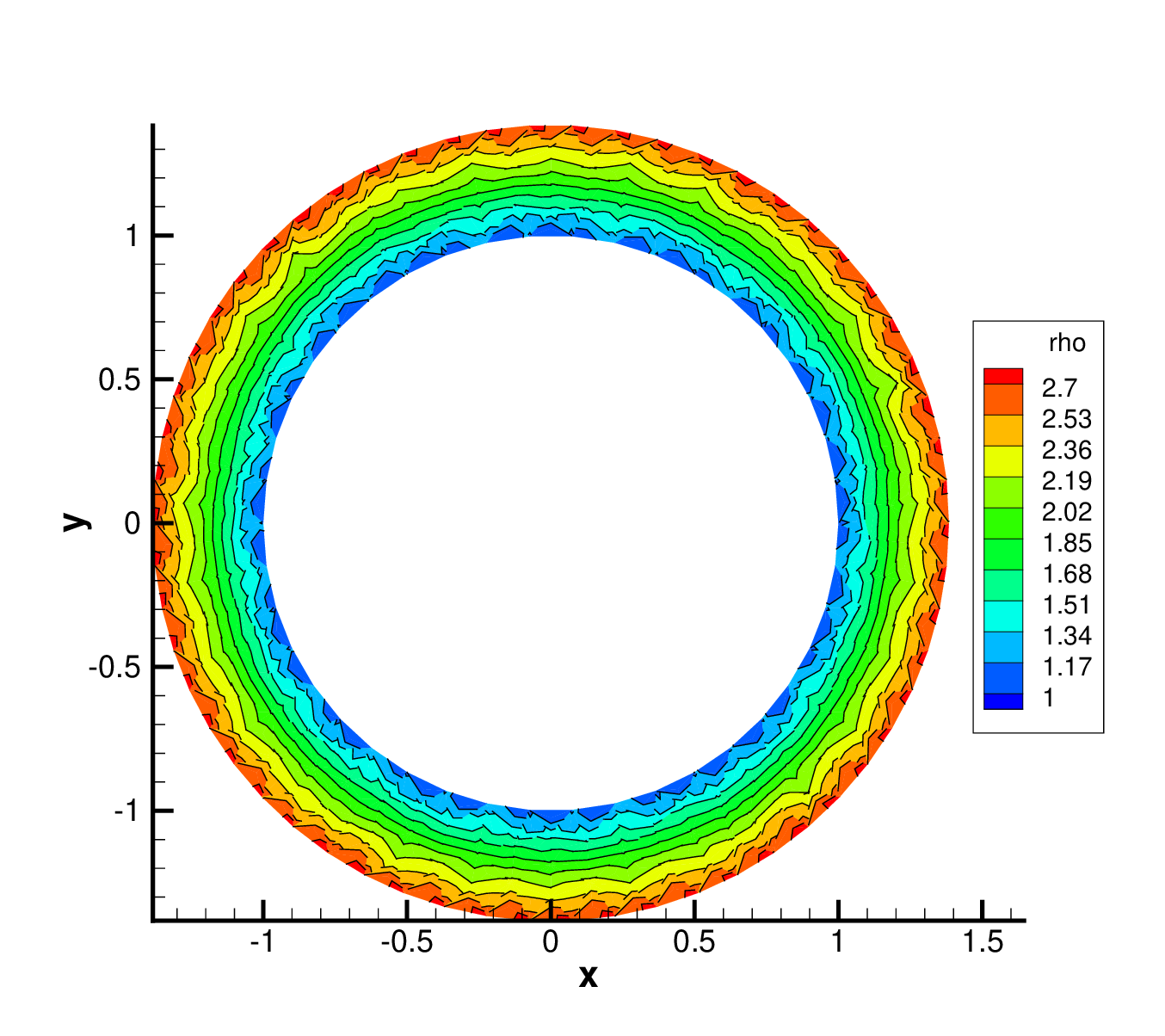}}
  \subfigure[]{\includegraphics[width=0.48\textwidth]{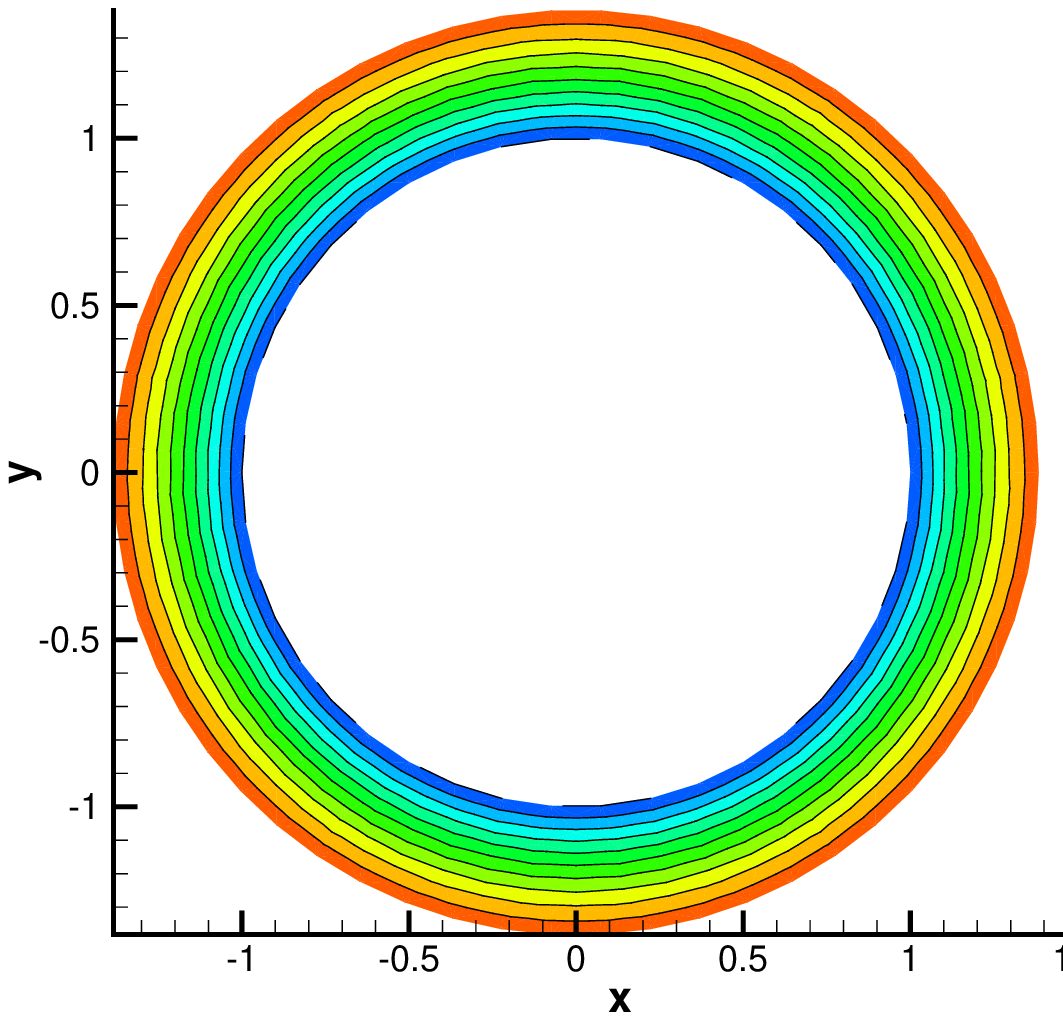}}
  \caption{Supersonic vortex bounded by two circular walls on 2D fixed meshes: DG-$\mathbb{P}_4$ numerical solution for the test case presented in~\cref{subsection:vortex}. Numerical solution obtained without (left) and with (right) the ROD-E polynomial correction.}\label{fig:vortex2D}
\end{figure}

\setlength{\tabcolsep}{3pt}  % default è 6pt
\begin{table}
  \caption{Supersonic vortex bounded by two circular walls on 2D fixed meshes: convergence analysis at $t=0.1$ for the test case presented in~\cref{subsection:vortex}.
  Numerical results obtained with the RK-DG framework and slip-wall boundary conditions imposed with and w/o the ROD-E and ROD-$L^2$ polynomial corrections on fixed conformal linear meshes.}\label{tab:vortex2D}
  \footnotesize
  \centering
  \begin{tabular}{ccccccccccccc}
          \hline\hline          &\multicolumn{2}{c}{$\rho$} &\multicolumn{2}{c}{$\rho U_1$}   &\multicolumn{2}{c}{$\rho$} &\multicolumn{2}{c}{$\rho U_1$}  &\multicolumn{2}{c}{$\rho$} &\multicolumn{2}{c}{$\rho U_1$}  \\[0.5mm]
          \cline{2-13}
          Grid size & Error        & Order & Error        & Order & Error      & Order  & Error      & Order & Error      & Order  & Error      & Order  \\[0.5mm]\hline
          &\multicolumn{12}{c}{DG-$\mathbb{P}_1$}\\
          &\multicolumn{4}{c}{w/o correction}   &\multicolumn{4}{c}{ROD-E}  &\multicolumn{4}{c}{ROD-$L^2$} \\[0.5mm]
7.52E-2   &  3.03E-02 &  --  & 4.15E-02 &  --  &  9.51E-03 &   -- & 1.84E-02 &  --  & 9.52E-03 &	 -- & 1.84E-02 &   -- \\
3.76E-2   &  1.03E-02 & 1.55 & 1.40E-02 & 1.56 &  2.27E-03 & 2.0  & 4.24E-03 & 2.11	& 2.27E-03 & 2.07	& 4.24E-03 & 2.11  \\  
1.88E-2   &  3.58E-03 & 1.52 & 4.87E-03 & 1.52 &  5.77E-04 & 1.9  & 1.06E-03 & 2.00	& 5.77E-04 & 1.97	& 1.06E-03 & 2.00  \\
9.40E-3   &  1.26E-03 & 1.51 & 1.71E-03 & 1.51 &  1.47E-04 & 1.9  & 2.67E-04 & 1.98	& 1.47E-04 & 1.97	& 2.67E-04 & 1.98  \\
          &\multicolumn{12}{c}{DG-$\mathbb{P}_2$}\\
          &\multicolumn{4}{c}{w/o correction}   &\multicolumn{4}{c}{ROD-E}  &\multicolumn{4}{c}{ROD-$L^2$} \\[0.5mm]
7.52E-2   & 4.52E-02 & 	--  & 5.55E-02 &  --  &  7.44E-04 & 	-- & 1.26E-03 &   -- &  7.34E-04 &  --  & 1.24E-03 &  --  \\
3.76E-2   & 1.78E-02 & 1.34 & 2.10E-02 & 1.39 &  1.15E-04 & 2.69 & 1.91E-04 & 2.72 &  1.14E-04 & 2.68 & 1.89E-04 & 2.71  \\
1.88E-2   & 6.29E-03 & 1.49 & 7.58E-03 & 1.47 &  1.43E-05 & 3.00 & 2.34E-05 & 3.02 &  1.76E-05 & 2.69 & 2.75E-05 & 2.78  \\
9.40E-3   & 2.31E-03 & 1.44 & 2.70E-03 & 1.49 &  2.60E-06 & 2.46 & 3.87E-06 & 2.59 &  2.59E-06 & 2.76 & 3.86E-06 & 2.82  \\
          &\multicolumn{12}{c}{DG-$\mathbb{P}_3$}\\
          &\multicolumn{4}{c}{w/o correction}   &\multicolumn{4}{c}{ROD-E}  &\multicolumn{4}{c}{ROD-$L^2$} \\[0.5mm]
7.52E-2   & 1.07E-01 & 	--  & 9.58E-02 &  --  & 2.80E-05 & 	--  & 4.88E-05 &  --  & 2.72E-05 & 	--  & 4.78E-05 & 	--  \\
3.76E-2   & 5.76E-02 & 0.89 & 4.73E-02 & 1.01 & 1.74E-06 & 4.00 & 3.08E-06 & 3.98 & 1.73E-06 & 3.97 & 3.06E-06 & 3.96 \\
1.88E-2   & 2.68E-02 & 1.10 & 2.25E-02 & 1.07 & 1.07E-07 & 4.02 & 1.88E-07 & 4.03 & 1.07E-07 & 4.02 & 1.87E-07 & 4.03 \\
9.40E-3   & 1.27E-02 & 1.08 & 1.05E-02 & 1.10 & 6.41E-09 & 4.05 & 1.12E-08 & 4.06 & 6.41E-09 & 4.05 & 1.12E-08 & 4.06 \\
          &\multicolumn{12}{c}{DG-$\mathbb{P}_4$}\\
          &\multicolumn{4}{c}{w/o correction}   &\multicolumn{4}{c}{ROD-E}  &\multicolumn{4}{c}{ROD-$L^2$} \\[0.5mm]
7.52E-2   & 1.19E-01 & 	--  & 1.13E-01 &  --  & 2.93E-06 & 	 -- & 6.30E-06 & 	--  & 2.89E-06 & 	--  & 6.24E-06 &  --  \\
3.76E-2   & 6.39E-02 & 0.89 & 5.65E-02 & 1.00 & 1.12E-07 & 4.70	& 2.43E-07 & 4.69 & 1.11E-07 & 4.70 & 2.42E-07 & 4.68  \\
1.88E-2   & 3.02E-02 & 1.08 & 2.70E-02 & 1.06 & 4.00E-09 & 4.80	& 8.55E-09 & 4.82 & 3.98E-09 & 4.80 & 8.52E-09 & 4.82  \\
9.40E-3   & 1.47E-02 & 1.04 & 1.28E-02 & 1.07 & 1.31E-10 & 4.92	& 2.77E-10 & 4.94 & 1.31E-10 & 4.92 & 2.71E-10 & 4.97  \\
          \hline\hline\\[1pt]
  \end{tabular}
  \end{table}

\section{Numerical experiments in the ADER-ALE FV framework}\label{section:ResultsFV}

\subsection{Manufactured solution on 2D and 3D moving meshes}\label{subsection:ManufacturedMoving}

The two-dimensional test case described in \cref{subsection:ManufacturedFixed} is extended to a moving domain to validate the new boundary corrections in the ADER-ALE FV framework in both two and three space dimensions. 
For the 2D case, we investigate the influence of the boundary correction on a moving mesh with a radial velocity field prescribed at each boundary node $\nu$, as follows:
\[
\mathbf{V}_{\nu} = V_0 \|\mathbf{x}_{\nu}\| \begin{bmatrix} \cos(\alpha_{\nu}) \\ \sin(\alpha_{\nu}) \end{bmatrix},
\]
where $V_0 = 0.1$ denotes the maximum velocity and $\alpha_{\nu} = \arctan(y_{\nu}, x_{\nu})$. Since the domain evolves in time and the solution is spatially dependent, the boundary condition becomes time-dependent, thereby providing a more challenging test case for high-order boundary conditions.
As observed for the numerical results in the RK-DG framework, also for this moving domain problem solved with the ADER-ALE FV methods, both ROD-E and ROD-$L^2$ recover all convergence slopes as shown in \cref{tab:manuf2Dmoving}, while the results without correction present an accuracy of second order for all the polynomial basis considered.

In order to test the proposed boundary conditions in 3D, we extend the previous two-dimensional test case to a spherical moving domain. The source term in the three-dimensional inhomogeneous Euler equations is defined as:
\begin{equation*}\label{eq:manufactured3D}
  \mathbf{s} = \begin{bmatrix}  0.6\cos(x+y+z) \\ 0.8\cos(x+y+z)  \\ 0.8\cos(x+y+z) \\ 0.8\cos(x+y+z) \\ 3.0 \cos(x+y+z) \end{bmatrix},
\end{equation*}
In this case, the exact steady state solution is given by
$$ \rho = 1 + 0.2 \sin(x+y+z),\quad U_1=1,\quad U_2=1, \quad U_3=1, \quad p=1 + 0.2 \sin(x+y+z),$$
which is imposed on the boundary as a ghost state.
To study the impact of the boundary correction on the moving mesh, we impose the following velocity field
on each boundary node $\nu$:
$$ \mathbf{V}_{\nu} = V_0 \frac{\mathbf{x}_{\nu}}{\|\mathbf{x}_{\nu}\|} ,$$
where $V_0=0.1$ is the maximum velocity.
Once again, the solution is space-dependent, and the domain is evolving, providing a moving test case for the proposed high order boundary conditions in the considered ALE framework. 
The evolution of the computational domain and solution in 2D and 3D is visible in \cref{fig:manuf2D3Dmoving}, where the initial and final mesh configurations are plotted.
As in the 2D case, also for the 3D moving manufactured solution, we observe in \cref{tab:manuf3Dmoving} that all convergence orders are achieved with both ROD-E and ROD-$L^2$ corrections, while only second-order is maintained when no correction is used. Due to limited computational resources, the 3D simulations are carried out up to fourth order of accuracy, instead of fifth order as in the 2D case.

\begin{figure}
  \centering
  \subfigure[]{\includegraphics[width=0.48\textwidth]{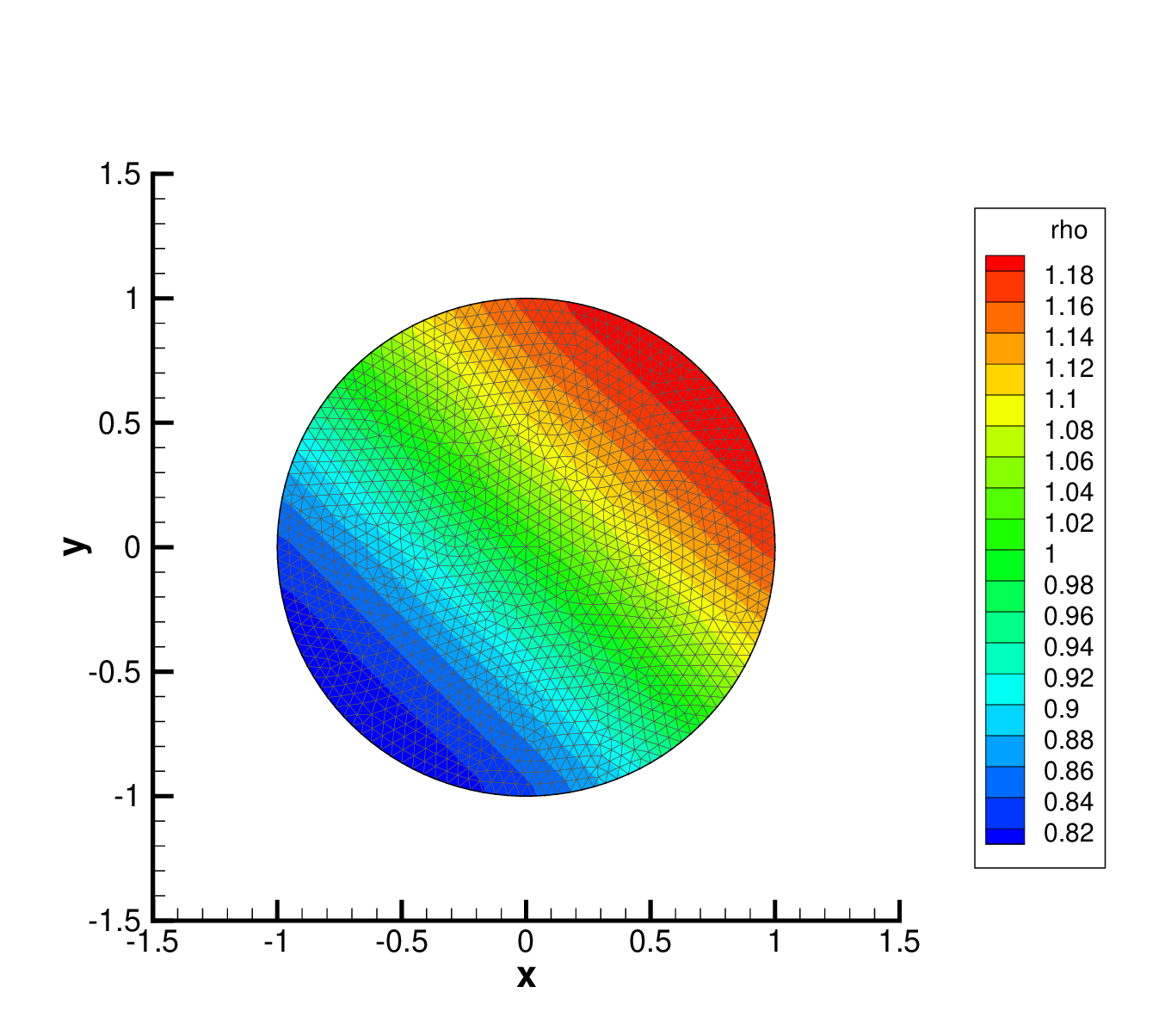}}
  \subfigure[]{\includegraphics[width=0.48\textwidth]{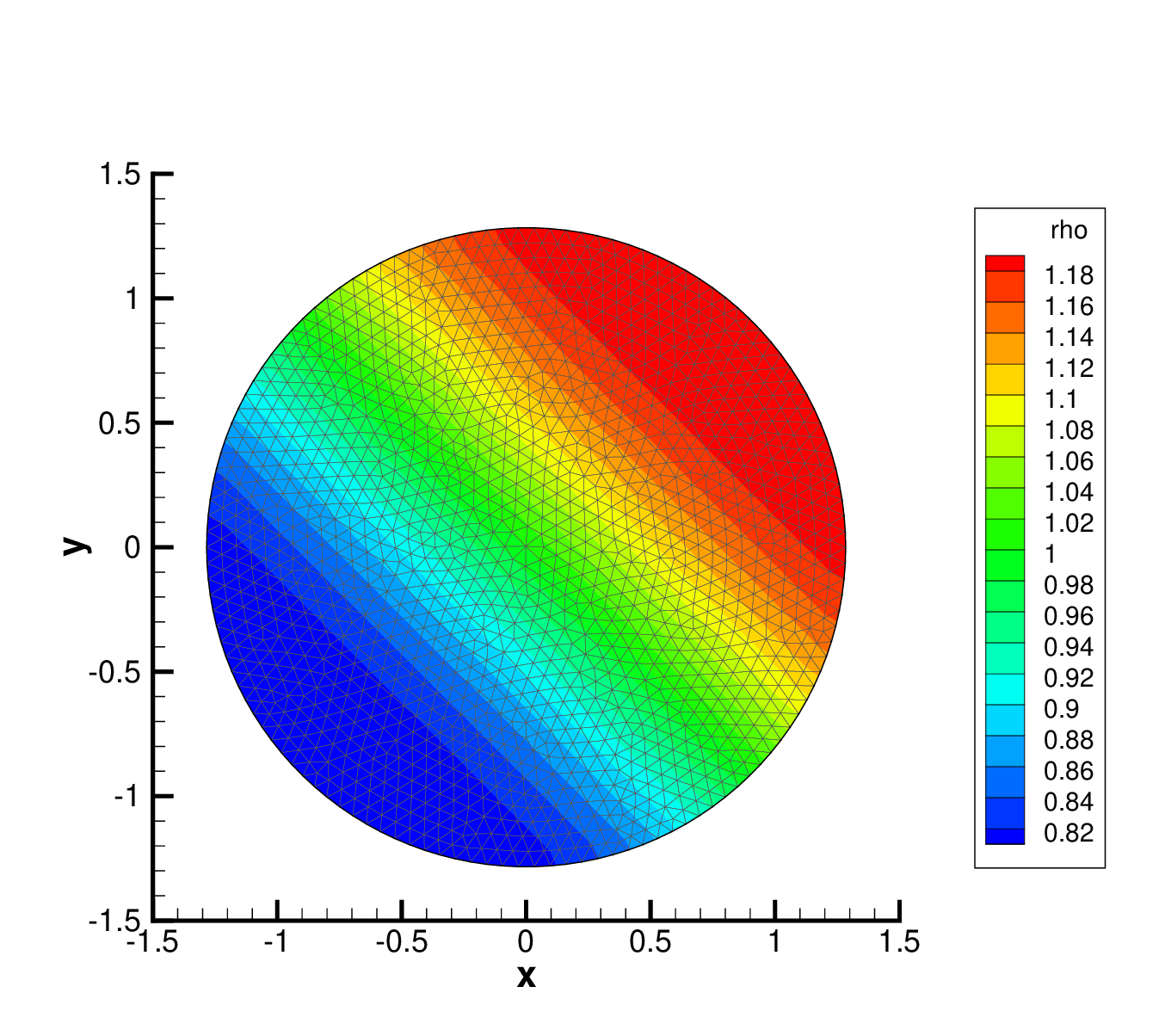}}
  \subfigure[]{\includegraphics[width=0.48\textwidth]{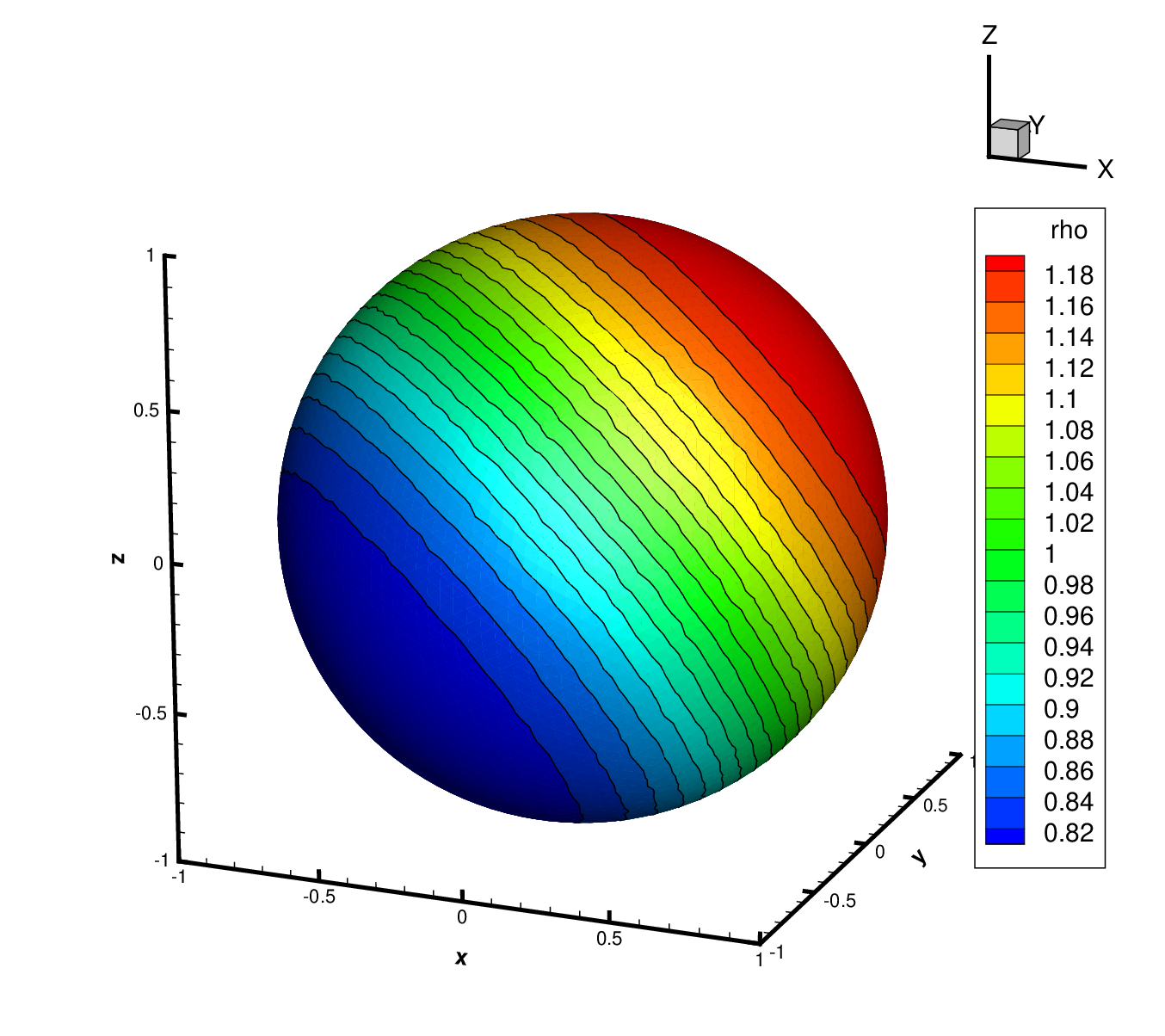}}
  \subfigure[]{\includegraphics[width=0.48\textwidth]{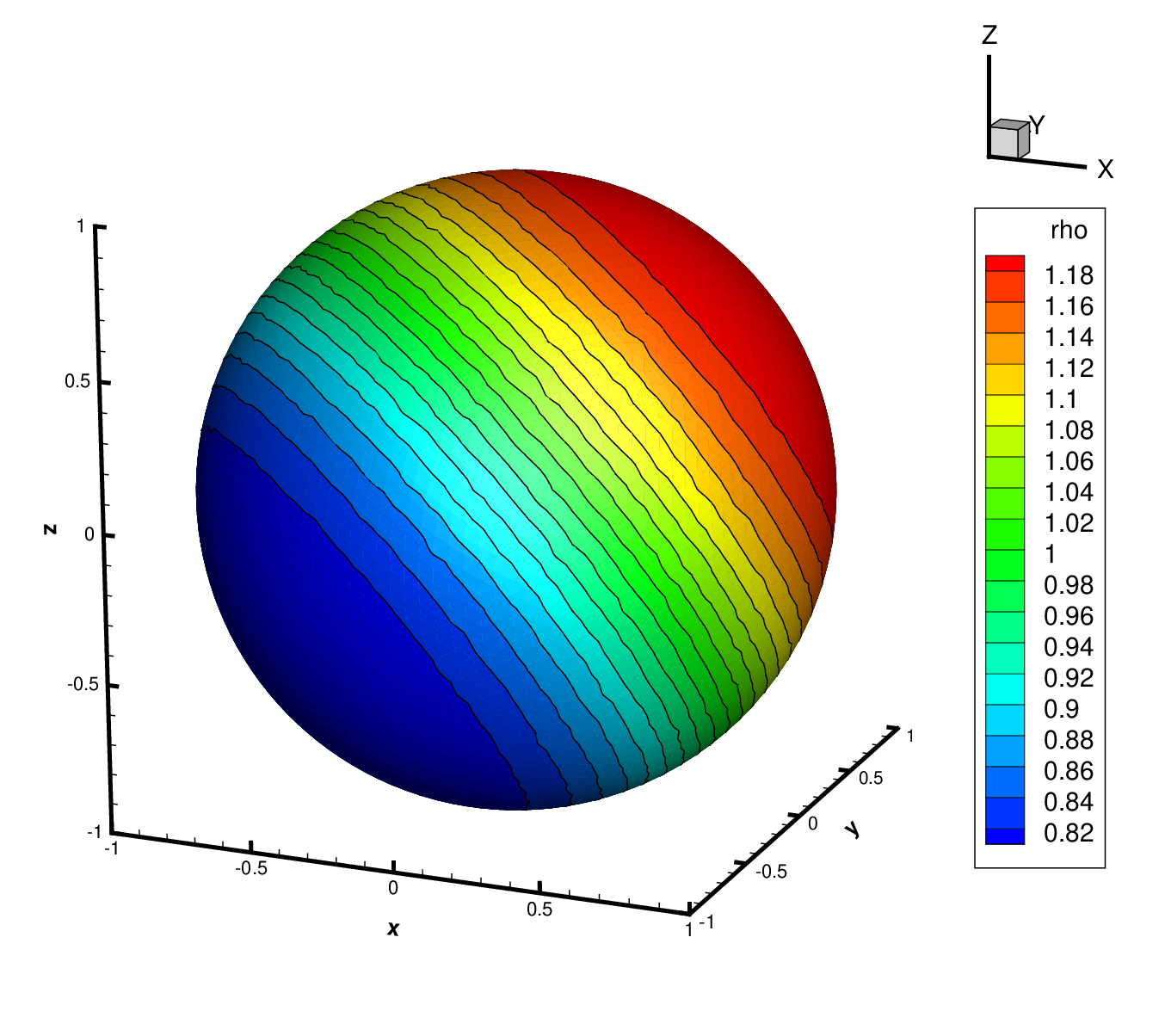}}
  \caption{Manufactured solution on 2D and 3D moving meshes: initial (left) and final (right) mesh configuration.
  The mesh is moving radially outward. The color map represents the density field.}
  \label{fig:manuf2D3Dmoving}
\end{figure}

\setlength{\tabcolsep}{3pt}  % default è 6pt
\begin{table}
  \caption{Manufactured solution on 2D moving meshes: convergence analysis at $t=0.1$ for the test case presented in~\cref{subsection:ManufacturedMoving}.
  Numerical results obtained with the ADER-ALE FV framework and Dirichlet-type boundary conditions imposed with and w/o the ROD-E and ROD-$L^2$ polynomial corrections on moving conformal linear meshes.}\label{tab:manuf2Dmoving}
  \footnotesize
  \centering
  \begin{tabular}{ccccccccccccc}
          \hline\hline          &\multicolumn{2}{c}{$\rho$} &\multicolumn{2}{c}{$U_1$}   &\multicolumn{2}{c}{$\rho$} &\multicolumn{2}{c}{$U_1$}  &\multicolumn{2}{c}{$\rho$} &\multicolumn{2}{c}{$U_1$}  \\[0.5mm]
          \cline{2-13}
          Grid size & Error        & Order & Error        & Order & Error      & Order  & Error      & Order & Error      & Order  & Error      & Order  \\[0.5mm]\hline
          &\multicolumn{12}{c}{FV-$\mathbb{P}_1$}\\
          &\multicolumn{4}{c}{w/o correction}   &\multicolumn{4}{c}{ROD-E}  &\multicolumn{4}{c}{ROD-$L^2$} \\[0.5mm]
1.90E-01  &  4.47E-03  &   --   &  2.00E-03   &  --  & 4.62E-03	& --   &  1.96E-03 &  --  & 4.60E-03 & 	--  & 1.96E-03 &  -- \\    
9.76E-02  &  1.27E-03  &  1.89  &  7.32E-04   & 1.51 & 1.29E-03	& 1.91 &	7.31E-04 & 1.48 & 1.28E-03 & 1.84 & 7.31E-04 & 1.42 \\ 
5.04E-02  &  3.59E-04  &  1.91  &  2.27E-04   & 1.77 & 3.61E-04	& 1.92 &	2.27E-04 & 1.76 & 3.61E-04 & 1.83 & 2.27E-04 & 1.68 \\ 
2.55E-02  &  9.90E-05  &  1.89  &  6.78E-05   & 1.77 & 9.93E-05	& 1.89 &	6.79E-05 & 1.77 & 9.92E-05 & 1.86 & 6.79E-05 & 1.74 \\           
          &\multicolumn{12}{c}{FV-$\mathbb{P}_2$}\\
          &\multicolumn{4}{c}{w/o correction}   &\multicolumn{4}{c}{ROD-E}  &\multicolumn{4}{c}{ROD-$L^2$} \\[0.5mm]
1.90E-01  &  6.83E-04  &   --   &  2.30E-04   & --   & 7.79E-04 &  --  & 2.53E-04	&  --  & 7.76E-04	&  --  & 2.53E-04	&  -- \\
9.76E-02  &  8.56E-05  &  3.11  &  4.32E-05   & 2.51 & 8.16E-05 & 3.38 & 2.49E-05	& 3.47 & 8.15E-05	& 3.38 & 2.49E-05	& 3.47 \\
5.04E-02  &  1.73E-05  &  2.42  &  1.11E-05   & 2.05 & 1.01E-05 & 3.15 & 3.60E-06	& 2.92 & 1.01E-05	& 3.15 & 3.60E-06	& 2.92 \\
2.55E-02  &  4.33E-06  &  2.03  &  2.84E-06   & 2.00 & 1.24E-06 & 3.08 & 4.64E-07	& 3.00 & 1.24E-06	& 3.08 & 4.64E-07	& 3.00 \\        
          &\multicolumn{12}{c}{FV-$\mathbb{P}_3$}\\
          &\multicolumn{4}{c}{w/o correction}   &\multicolumn{4}{c}{ROD-E}  &\multicolumn{4}{c}{ROD-$L^2$} \\[0.5mm]
1.90E-01  & 2.66E-04  &   --   &  1.96E-04   & --    & 9.91E-05	 & 	--   & 5.50E-05	& --   & 9.63E-05	 & --    & 5.80E-05 &  --   \\
9.76E-02  & 6.78E-05  &  2.05  &  4.56E-05   & 2.19  & 6.45E-06  & 4.10  & 4.08E-06 &	3.90 &  6.33E-06 &	4.08 & 4.14E-06 & 3.96    \\
5.04E-02  & 1.80E-05  &  2.00  &  1.18E-05   & 2.05  & 5.25E-07  & 3.79  & 3.92E-07 &	3.54 &  5.19E-07 &	3.78 & 3.94E-07 & 3.55    \\
2.55E-02  & 4.56E-06  &  2.01  &  2.97E-06   & 2.02  & 3.70E-08  & 3.89  & 3.32E-08 &	3.62 &  3.68E-08 &	3.88 & 3.32E-08 & 3.62    \\
          &\multicolumn{12}{c}{FV-$\mathbb{P}_4$}\\
          &\multicolumn{4}{c}{w/o correction}   &\multicolumn{4}{c}{ROD-E}  &\multicolumn{4}{c}{ROD-$L^2$} \\[0.5mm]
1.90E-01  & 2.69E-04 &   -- & 1.80E-04 &  --  & 2.48E-05 & --   & 1.07E-05 &  --  & 2.47E-05 & --   & 1.08E-05 &	 --  \\
9.76E-02  & 7.00E-05 & 2.02 & 4.55E-05 & 2.06 & 7.15E-07 & 5.32 & 3.22E-07 & 5.26 & 7.09E-07 & 5.32 & 3.22E-07 &	5.26 \\
5.04E-02  & 1.82E-05 & 2.03 & 1.18E-05 & 2.04 & 2.45E-08 & 5.10 & 1.58E-08 & 4.55 & 2.44E-08 & 5.09 & 1.58E-08 &	4.56 \\
2.55E-02  & 4.57E-06 & 2.02 & 2.97E-06 & 2.02 & 2.82E-09 & 3.17 & 1.13E-09 & 3.87 & 2.81E-09 & 3.17 & 1.12E-09 &	3.87 \\
          \hline\hline\\[1pt]
  \end{tabular}
  \end{table}

\setlength{\tabcolsep}{3pt}  % default è 6pt
\begin{table}
  \caption{Manufactured solution on 3D moving meshes: convergence analysis at $t=0.1$ for the test case presented in~\cref{subsection:ManufacturedMoving}.
  Numerical results obtained with the ADER-ALE FV framework and Dirichlet-type boundary conditions imposed with and w/o the ROD-E and ROD-$L^2$ polynomial corrections on moving conformal linear meshes.}\label{tab:manuf3Dmoving}
  \footnotesize
  \centering
  \begin{tabular}{ccccccccccccc}
          \hline\hline          &\multicolumn{2}{c}{$\rho$} &\multicolumn{2}{c}{$U_1$}   &\multicolumn{2}{c}{$\rho$} &\multicolumn{2}{c}{$U_1$}  &\multicolumn{2}{c}{$\rho$} &\multicolumn{2}{c}{$U_1$}  \\[0.5mm]
          \cline{2-13}
          Grid size & Error        & Order & Error        & Order & Error      & Order  & Error      & Order & Error      & Order  & Error      & Order  \\[0.5mm]\hline
          &\multicolumn{12}{c}{FV-$\mathbb{P}_1$}\\
          &\multicolumn{4}{c}{w/o correction}   &\multicolumn{4}{c}{ROD-E}  &\multicolumn{4}{c}{ROD-$L^2$} \\[0.5mm]
2.56E-01  &  3.44E-03  &   --  &  1.13E-03   &  --  &  4.15E-03 &  --  & 9.03E-04 &  --  &  4.07E-03 & 	--  & 9.23E-04 &  --   \\
1.45E-01  &  1.27E-03  & 1.75  &  4.69E-04   & 1.55 &  1.38E-03 & 1.64 & 4.23E-04 & 1.13 &  1.38E-03 & 1.62 & 4.26E-04 & 1.16  \\          
7.75E-02  &  3.52E-04  & 2.06  &  1.56E-04   & 1.76 &  3.75E-04 & 1.97 & 1.48E-04 & 1.58 &  3.74E-04 & 1.97 & 1.48E-04 & 1.59  \\
3.91E-02  &  8.99E-05  & 2.00  &  4.54E-05   & 1.80 &  9.46E-05 & 2.02 & 4.42E-05 & 1.77 &  9.45E-05 & 2.01 & 4.42E-05 & 1.77  \\
          &\multicolumn{12}{c}{FV-$\mathbb{P}_2$}\\
          &\multicolumn{4}{c}{w/o correction}   &\multicolumn{4}{c}{ROD-E}  &\multicolumn{4}{c}{ROD-$L^2$} \\[0.5mm]
2.56E-01  &  2.02E-03  &  --   &  4.40E-04   &   --  &  2.14E-03 & --	  & 2.48E-04 &  --  & 2.15E-03 & 	--  & 2.58E-04 &  --    \\
1.45E-01  &  4.09E-04  & 2.80  &  1.21E-04   &  2.27 &  3.40E-04 & 2.76	& 4.93E-05 & 2.42 & 3.41E-04 & 2.76	& 4.99E-05 & 2.46    \\
7.75E-02  &  8.38E-05  & 2.54  &  3.24E-05   &  2.11 &  4.43E-05 & 3.08	& 8.17E-06 & 2.71 & 4.43E-05 & 3.08	& 8.20E-06 & 2.73    \\
3.91E-02  &  1.94E-05  & 2.14  &  8.28E-06   &  2.00 &  5.07E-06 & 3.17	& 1.15E-06 & 2.87 & 5.07E-06 & 3.18	& 1.15E-06 & 2.88    \\          
          &\multicolumn{12}{c}{FV-$\mathbb{P}_3$}\\
          &\multicolumn{4}{c}{w/o correction}   &\multicolumn{4}{c}{ROD-E}  &\multicolumn{4}{c}{ROD-$L^2$} \\[0.5mm]
2.56E-01  &  1.41E-03  &   --  &  4.12E-04   &  --   &  2.72E-04 & 	--  & 5.39E-05 &  --  & 2.65E-04 & 	--  & 6.27E-05 &  --    \\
1.45E-01  &  3.30E-04  & 2.55  &  1.28E-04   & 2.06  &  2.65E-05 & 3.49 & 6.95E-06 & 3.07 & 2.63E-05 & 3.46 & 7.29E-06 & 3.23   \\
7.75E-02  &  7.95E-05  & 2.28  &  3.38E-05   & 2.13  &  1.95E-06 & 3.94 & 6.24E-07 & 3.64 & 1.94E-06 & 3.94 & 6.31E-07 & 3.70   \\
3.91E-02  &  1.95E-05  & 2.06  &  8.53E-06   & 2.02  &  1.15E-07 & 4.15 & 4.70E-08 & 3.79 & 1.15E-07 & 4.14 & 4.72E-08 & 3.80   \\          
          \hline\hline\\[1pt]
  \end{tabular}
  \end{table}

\subsection{Kidder problem in 2D}\label{subsection:Kidder2D}

The so-called Kidder problem describes the isentropic compression of a shell filled with an ideal gas. 
An exact analytical solution was originally proposed by Kidder in \cite{kidder1976laser}. 
This problem is widely used as a benchmark for Lagrangian hydrodynamics codes to verify the absence 
of spurious numerical artifacts. The shell features time-dependent inner and outer radii, denoted 
$r_i(t)$ and $r_e(t)$, respectively. Their initial values are $r_i(0)=r_{i,0}=0.9$ and $r_e(0)=r_{e,0}=1$. 
The specific heat ratio is $\gamma=2$, and the initial density distribution is given by
$$ \rho_0(r) = \rho(r,0) = \left( \frac{r^2_{e,0}-r^2}{r^2_{e,0}-r^2_{i,0}}\rho_{i,0}^{\gamma-1} + 
\frac{r^2-r^2_{i,0}}{r^2_{e,0}-r^2_{i,0}}\rho_{e,0}^{\gamma-1} \right)^{\frac{1}{\gamma-1}},\qquad\text{with}\quad r_i(t)\leq r \leq r_e(t), $$
where we set $\rho_{i,0}=1$ and $\rho_{e,0}=2$, representing the initial densities at the inner and outer boundaries, respectively.
The initial entropy $s_0=\frac{p_0}{\rho_0^\gamma}=1$ is uniform, so that the initial pressure field becomes
$p_0(r)= s_0 \rho_0(r)^{\gamma}$. Initially, the fluid is at rest ($U_1=U_2\equiv 0$).

The time-dependent solution of the Kidder problem can be expressed in a self-similar form $R(r,t)=h(t) r$, 
where $R(r,t)$ denotes the radius at time $t>0$ of a fluid particle initially located at $r$. 
Consequently, for $t\in[0,\tau]$, the solution reads
\begin{align*}
\rho(R(r,t),t) &=   h(t)^{-\frac{2}{\gamma-1}} \rho_0\left(\frac{R(r,t)}{h(t)}\right),\\
U_r(R(r,t),t)  &= \frac{R(r,t)}{h(t)} \frac{\diff{}}{\diff{t}}h(t), \\
p(R(r,t),t)    &=  h(t)^{-\frac{2\gamma}{\gamma-1}}p_0\left(\frac{R(r,t)}{h(t)}\right),
\end{align*}
where $U_r$ denotes the radial velocity component.\\
The scaling factor (homothety rate) $h(t)$ and the focalisation time $\tau$ are defined as
$$ h(t) = \sqrt{1-\frac{t^2}{\tau^2}}, \qquad \tau = \sqrt{\frac{\gamma-1}{2} \frac{r^2_{e,0}-r^2_{i,0}}{c^2_{e,0}-c^2_{i,0}}}, $$
and the initial sound speeds at the inner and outer boundaries are
$$ c_{i,0} = \sqrt{\gamma\frac{ p_{i,0}}{\rho_{i,0}}}\quad \text{ and } \quad c_{e,0} = \sqrt{\gamma\frac{p_{e,0}}{\rho_{e,0}}}. $$

Boundary conditions are enforced on both the inner and outer surfaces of the shell, where the exact solution 
is imposed as a ghost state. The exact mesh velocity is prescribed at the boundary nodes.

Following \cite{carre2009cell}, the final simulation time is chosen as $t_f=\frac{\sqrt{3}}{2}\tau$, 
which yields a compression rate of $h(t_f)=1/2$ at the final time. Consequently, the final shell lies 
between radii 0.45 and 0.5. The initial and final mesh configurations are displayed in \cref{fig:Kidder2D}. 
Note that although simulations are performed on the full circular domain, only the top-right quarter is shown 
for better visualization.

We present the convergence analysis obtained for the Kidder problem in 2D in \cref{tab:Kidder2D}.
The numerical evidence shows that the high-order boundary corrections are essential to preserve the designed accuracy of the scheme, while a naive imposition of boundary data severely deteriorates the convergence rates. Without any correction, the observed convergence rates remain approximately of second order for all variables and polynomial degrees, rendering high-order approximations essentially useless. In contrast, the ROD-E and ROD-$L^2$ polynomial corrections dramatically improve the solution quality, producing errors on the finest mesh that are several orders of magnitude lower than those from the uncorrected formulation, with almost no additional computational cost.

\begin{figure}
  \centering
  \subfigure[]{\includegraphics[width=0.48\textwidth]{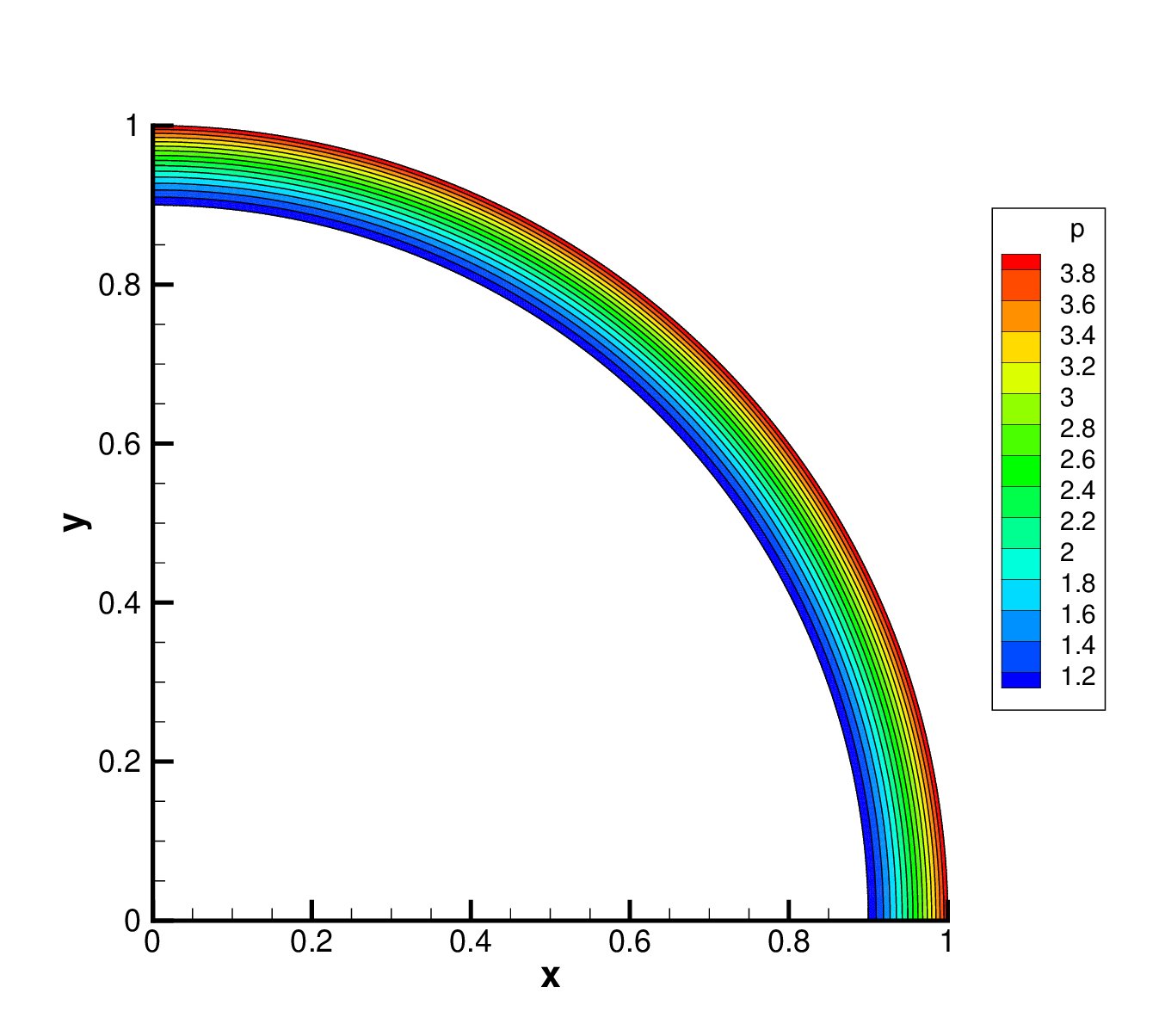}}
  \subfigure[]{\includegraphics[width=0.48\textwidth]{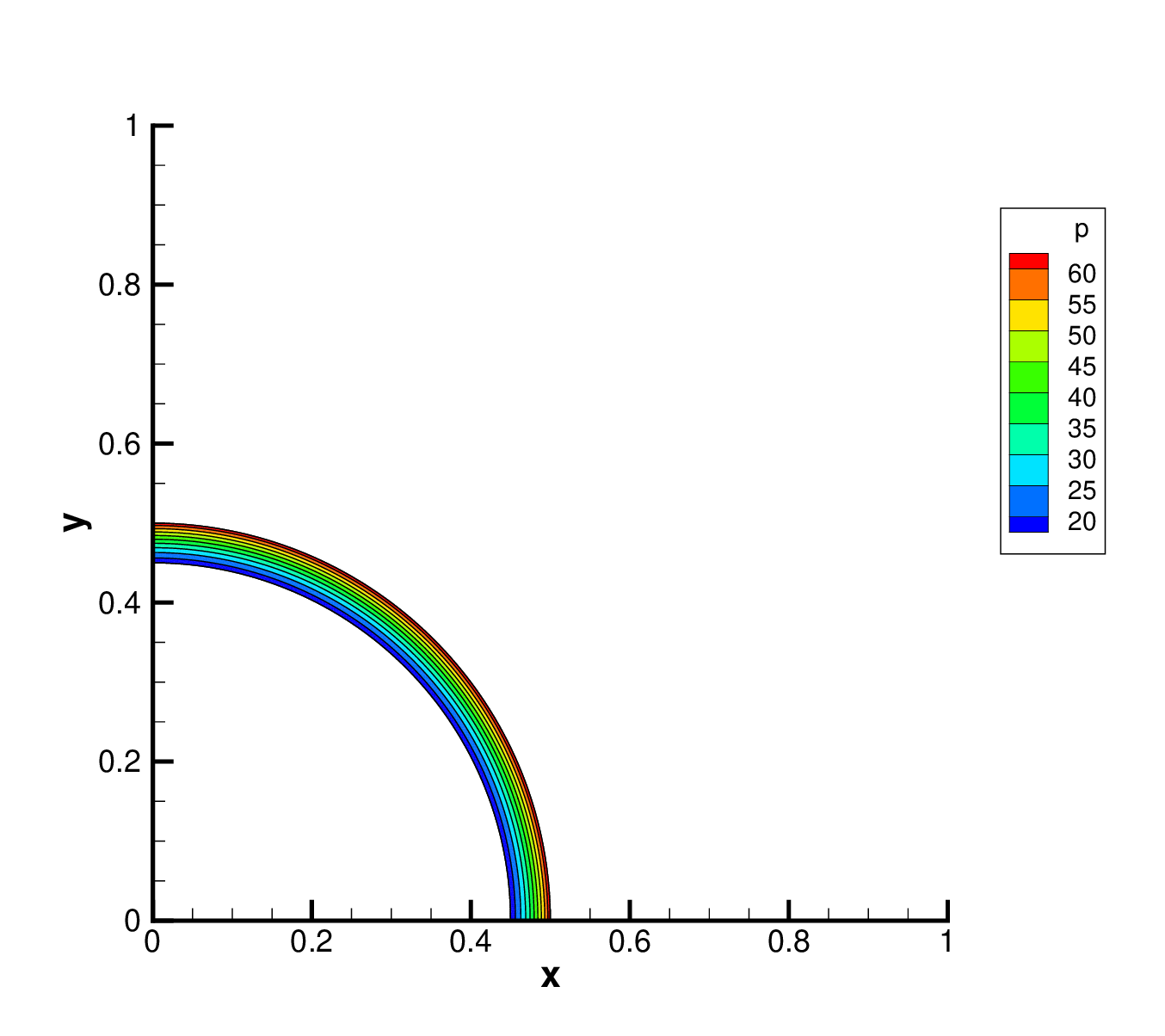}}
  \caption{Kidder problem in 2D: initial (left) and final (right) mesh configuration.
  The mesh is moving radially with a velocity field that depends on the Kidder exact solution.
  The color map represents the pressure field.}
  \label{fig:Kidder2D}
\end{figure}

\setlength{\tabcolsep}{3pt}  % default è 6pt
\begin{table}
  \caption{Kidder problem in 2D: convergence analysis for the test case presented in~\cref{subsection:Kidder2D}. 
  Numerical results obtained with the ADER-ALE FV framework and Dirichlet-type boundary conditions imposed with and w/o the ROD-E and ROD-$L^2$ polynomial corrections on moving conformal linear meshes.}\label{tab:Kidder2D}
  \footnotesize
  \centering
  \begin{tabular}{ccccccccccccc}
          \hline\hline          &\multicolumn{2}{c}{$\rho$} &\multicolumn{2}{c}{$U_1$}   &\multicolumn{2}{c}{$\rho$} &\multicolumn{2}{c}{$U_1$}  &\multicolumn{2}{c}{$\rho$} &\multicolumn{2}{c}{$U_1$}  \\[0.5mm]
          \cline{2-13}
          Grid size & Error        & Order & Error        & Order & Error      & Order  & Error      & Order & Error      & Order  & Error      & Order  \\[0.5mm]\hline
          &\multicolumn{12}{c}{FV-$\mathbb{P}_1$}\\
          &\multicolumn{4}{c}{w/o correction}   &\multicolumn{4}{c}{ROD-E}  &\multicolumn{4}{c}{ROD-$L^2$} \\[0.5mm]
1.28E-02  &  3.56E-02 & --   & 5.54E-03 & --    & 3.54E-02 & 	--  & 5.54E-03 &  --  & 3.54E-02 & 	--  & 5.54E-03 &  --  \\
5.99E-03  &  1.28E-02 & 1.35 & 1.57E-03 & 1.66  & 1.27E-02 & 1.34 & 1.57E-03 & 1.66 & 1.27E-02 & 1.34 & 1.57E-03 & 1.66 \\          
3.06E-03  &  2.97E-03 & 2.17 & 3.48E-04 & 2.24  & 2.98E-03 & 2.16 & 3.48E-04 & 2.23 & 2.98E-03 & 2.16 & 3.48E-04 & 2.23 \\
1.49E-03  &  7.35E-04 & 1.95 & 7.93E-05 & 2.06  & 7.38E-04 & 1.94 & 7.95E-05 & 2.05 & 7.38E-04 & 1.94 & 7.95E-05 & 2.05 \\
          &\multicolumn{12}{c}{FV-$\mathbb{P}_2$}\\
          &\multicolumn{4}{c}{w/o correction}   &\multicolumn{4}{c}{ROD-E}  &\multicolumn{4}{c}{ROD-$L^2$} \\[0.5mm]
1.28E-02  &  3.46E-03 & --   & 3.33E-04 & --    &  2.73E-03 & 	-- & 3.20E-04 & --   & 2.73E-03 & 	-- & 3.20E-04 &  --    \\
5.99E-03  &  6.35E-04 & 2.23 & 5.48E-05 & 2.38  &  3.83E-04 & 2.58 & 4.07E-05 & 2.71 & 3.83E-04 & 2.58 & 4.07E-05 & 2.71     \\
3.06E-03  &  1.20E-04 & 2.48 & 1.09E-05 & 2.40  &  4.05E-05 & 3.34 & 4.08E-06 & 3.42 & 4.04E-05 & 3.34 & 4.08E-06 & 3.42     \\
1.49E-03  &  2.54E-05 & 2.16 & 2.62E-06 & 1.99  &  4.34E-06 & 3.10 & 3.93E-07 & 3.25 & 4.34E-06 & 3.10 & 3.93E-07 & 3.25     \\          
          &\multicolumn{12}{c}{FV-$\mathbb{P}_3$}\\
          &\multicolumn{4}{c}{w/o correction}   &\multicolumn{4}{c}{ROD-E}  &\multicolumn{4}{c}{ROD-$L^2$} \\[0.5mm]
1.28E-02  &  1.39E-03 & --   & 1.68E-04 & --    &  1.74E-04 & --   & 4.02E-05 &  --  & 1.74E-04 &  --  & 4.02E-05 &  --    \\
5.99E-03  &  3.47E-04 & 1.82 & 4.22E-05 & 1.82  &  9.65E-06 & 3.80 & 2.36E-06 & 3.73 & 9.65E-06 & 3.80 & 2.36E-06 & 3.73     \\
3.06E-03  &  8.73E-05 & 2.05 & 1.06E-05 & 2.06  &  4.26E-07 & 4.64 & 1.23E-07 & 4.39 & 4.26E-07 & 4.64 & 1.23E-07 & 4.39     \\
1.49E-03  &  2.17E-05 & 1.94 & 2.63E-06 & 1.94  &  1.92E-08 & 4.30 & 5.50E-09 & 4.32 & 1.92E-08 & 4.30 & 5.50E-09 & 4.31     \\          
          \hline\hline\\[1pt]
  \end{tabular}
  \end{table}

\subsection{Horizontal oscillating cylinder}\label{subsection:FSIcylinder}

The last numerical experiment performed in this work is a two-dimensional fluid-body interaction test case, which involves a horizontal oscillating cylinder immersed in a fluid domain. We focus on the impact of the proposed boundary correction on the qualitative results of the simulation. In particular, we compare the results obtained with and without the ROD-$L^2$ correction, while the ROD-E correction is not considered, for conciseness.
The horizontal oscillating cylinder problem is set up in a square domain of size $[-10,10]^2$ with an internal cylinder of radius $R=1$. The initial mesh of $N_E=7145$ triangles is shown in \cref{fig:OC-mesh}, with a close-up view around the cylinder.
The test case involves a circular cylinder undergoing a prescribed horizontal harmonic oscillation, starting from a fluid initially at rest. The motion is defined as \mbox{$x(t) = -A \sin(2\pi f t)$},
where $A$ and $f$ denote the amplitude and frequency of the oscillation, respectively. The parameters are set to $A=0.1$ and $f=0.1$, and the simulation is run up to the final time $t_f=10$, corresponding to one complete period.

Figure \ref{fig:OC-H-velU} displays the velocity field $u$ at selected time instants for the horizontally oscillating cylinder, offering a side-by-side comparison between the results obtained with and without the ROD-$L^2$ correction. Both simulations are run with a third-order method, which is enough to capture the impact of the boundary correction. 
The uncorrected results, shown in the left column, exhibit significant spurious disturbances near the boundary, where the geometric error clearly dominates the numerical solution and produces a thick layer of inaccuracies. In contrast, the corrected results, presented in the right column, show a much smoother and more symmetric velocity field. 
The same improvement is evident in the entropy distribution $S=p/\rho^\gamma$, reported in  \cref{fig:OC-H-entropy}. 
As a matter of fact, the entropy field should be uniform in the entire domain, which allows to visually assess the presence of spurious entropy production. 
The left column reveals a pronounced spurious entropy errors when no correction is applied, while the right column shows that the correction effectively suppresses these artefacts, achieving a more accurate representation. To ensure a fair and direct visual comparison, the same color scale is used for both corrected and uncorrected results across all figures.

\begin{figure}
  \centering
  \subfigure[]{\includegraphics[width=0.4\textwidth]{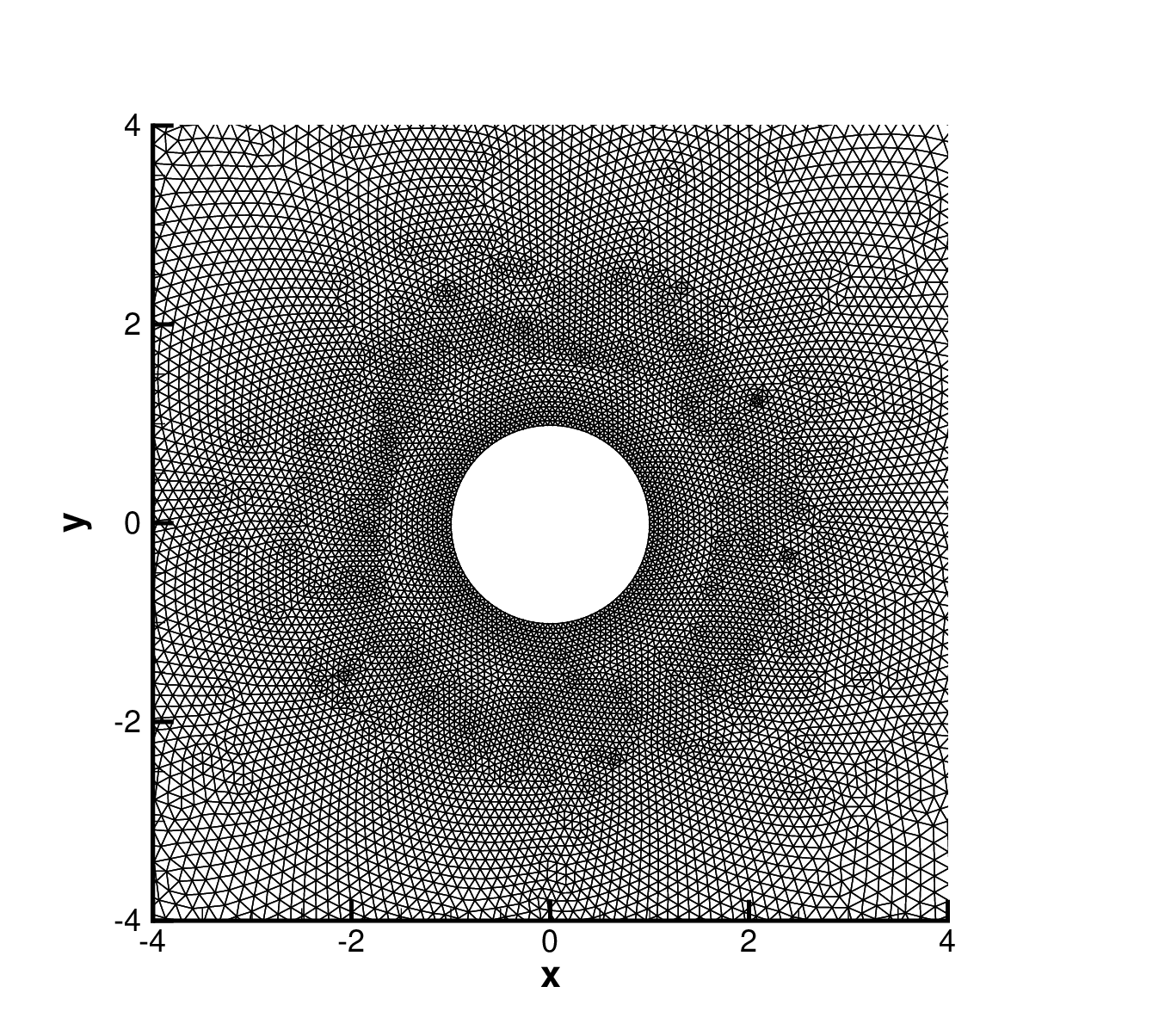}}
  \subfigure[]{\includegraphics[width=0.4\textwidth]{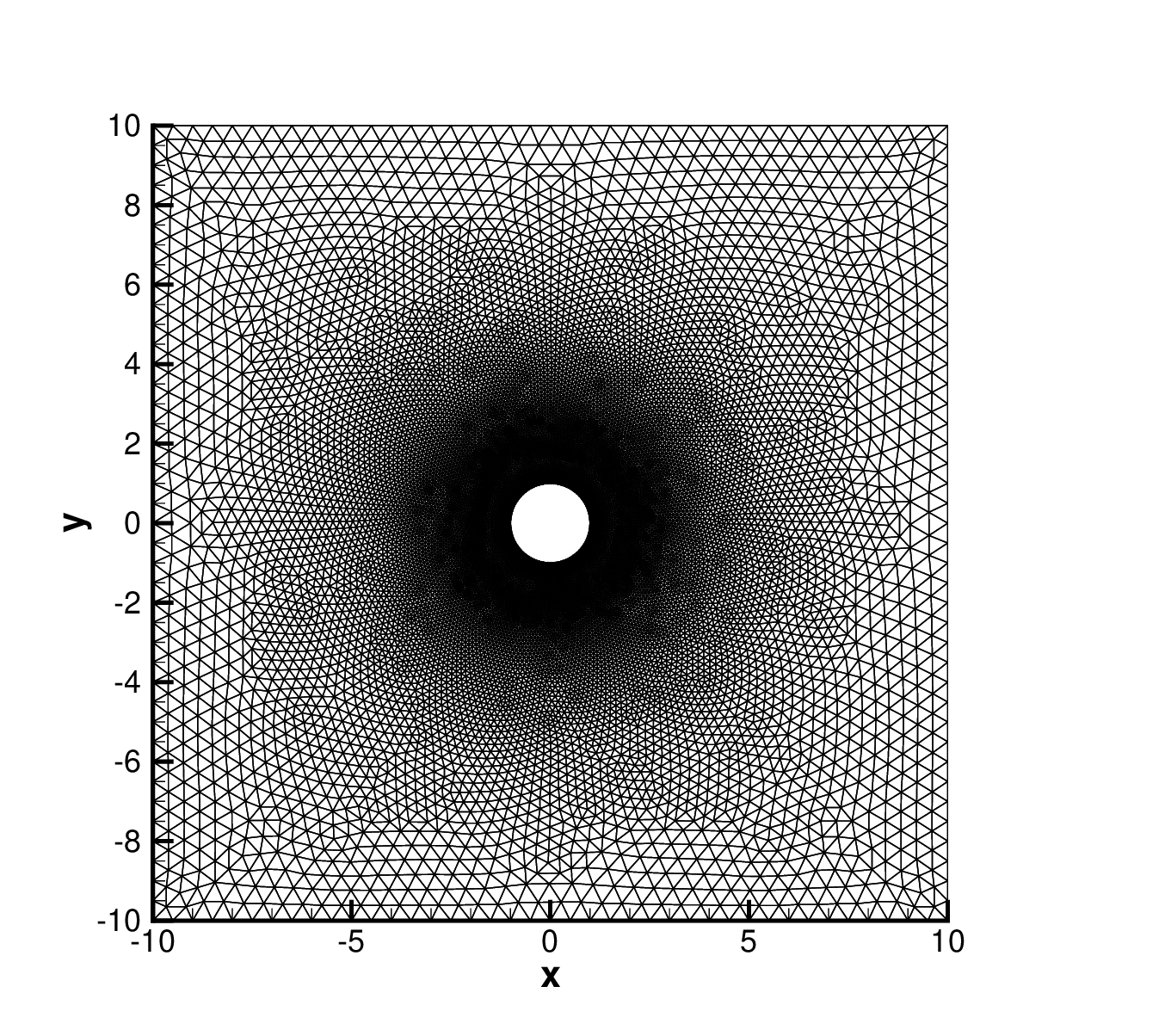}}
  \caption{Horizontal oscillating cylinder: initial mesh configuration (right) and a zoom close to the cylinder (left).}
  \label{fig:OC-mesh}
\end{figure}

\begin{figure}
  \centering
  \subfigure[$t=0.5$ s]{\includegraphics[width=0.4\textwidth]{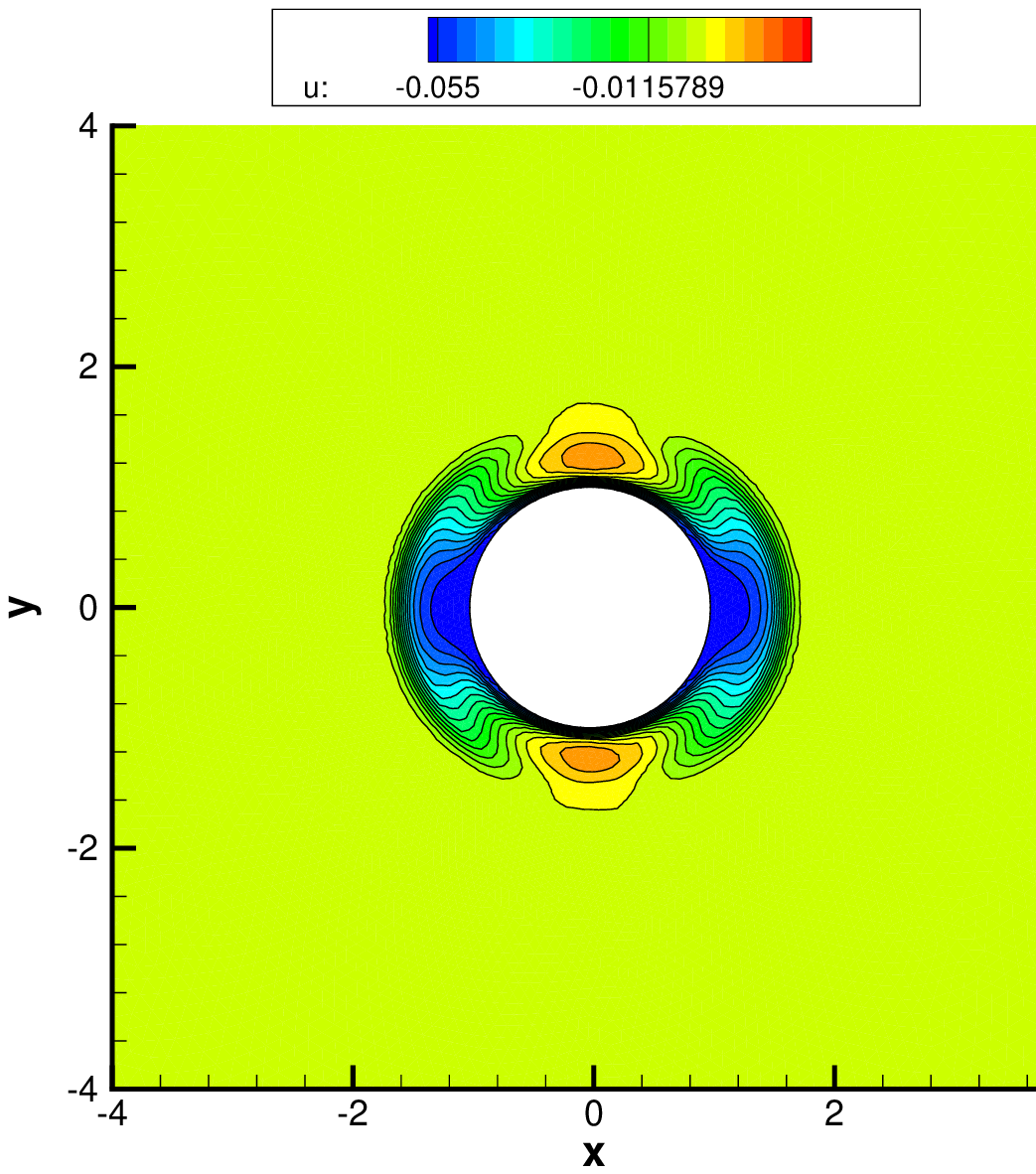}}
  \subfigure[$t=0.5$ s]{\includegraphics[width=0.4\textwidth]{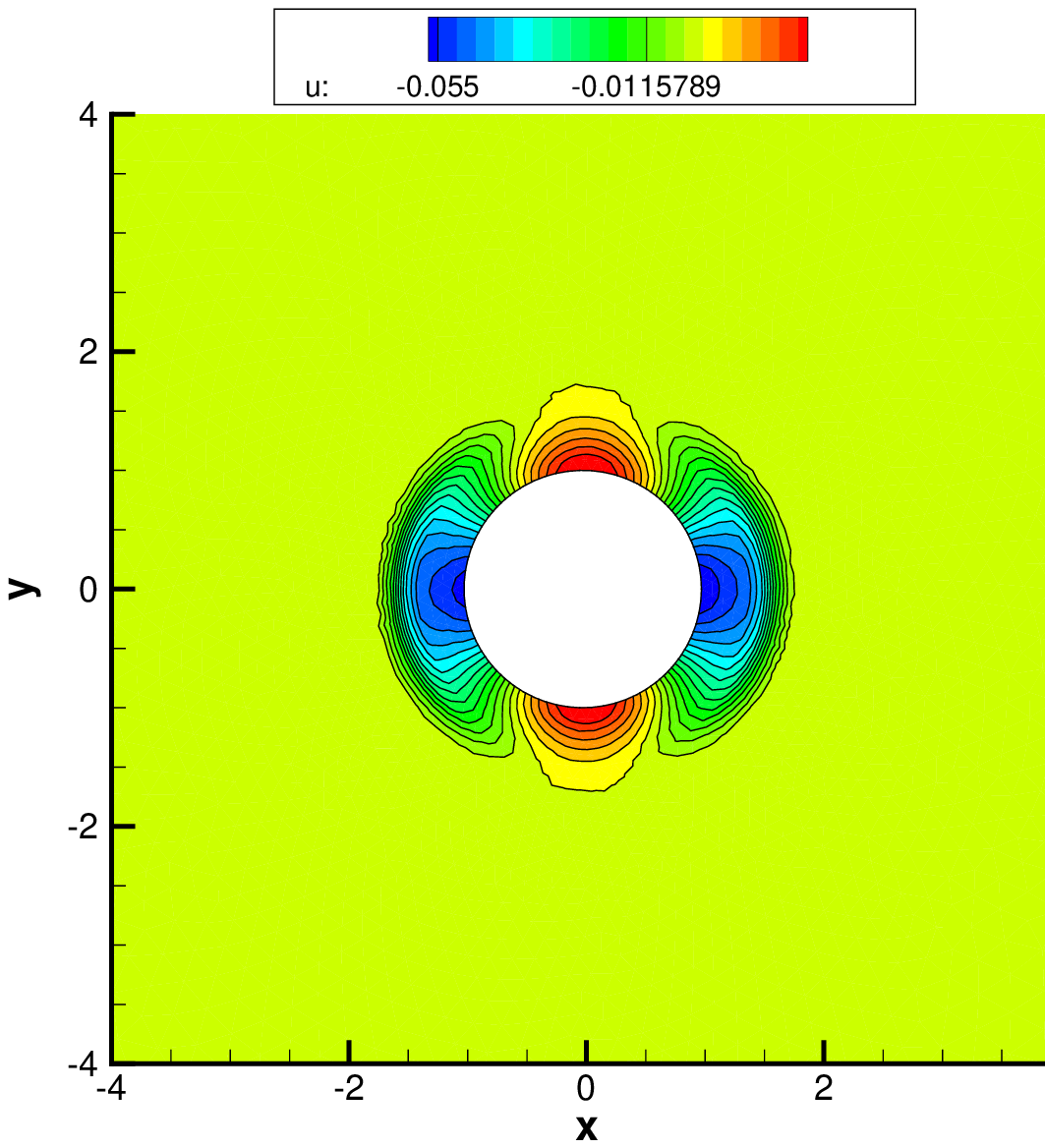}}
  \subfigure[$t=5$ s]{\includegraphics[width=0.4\textwidth]{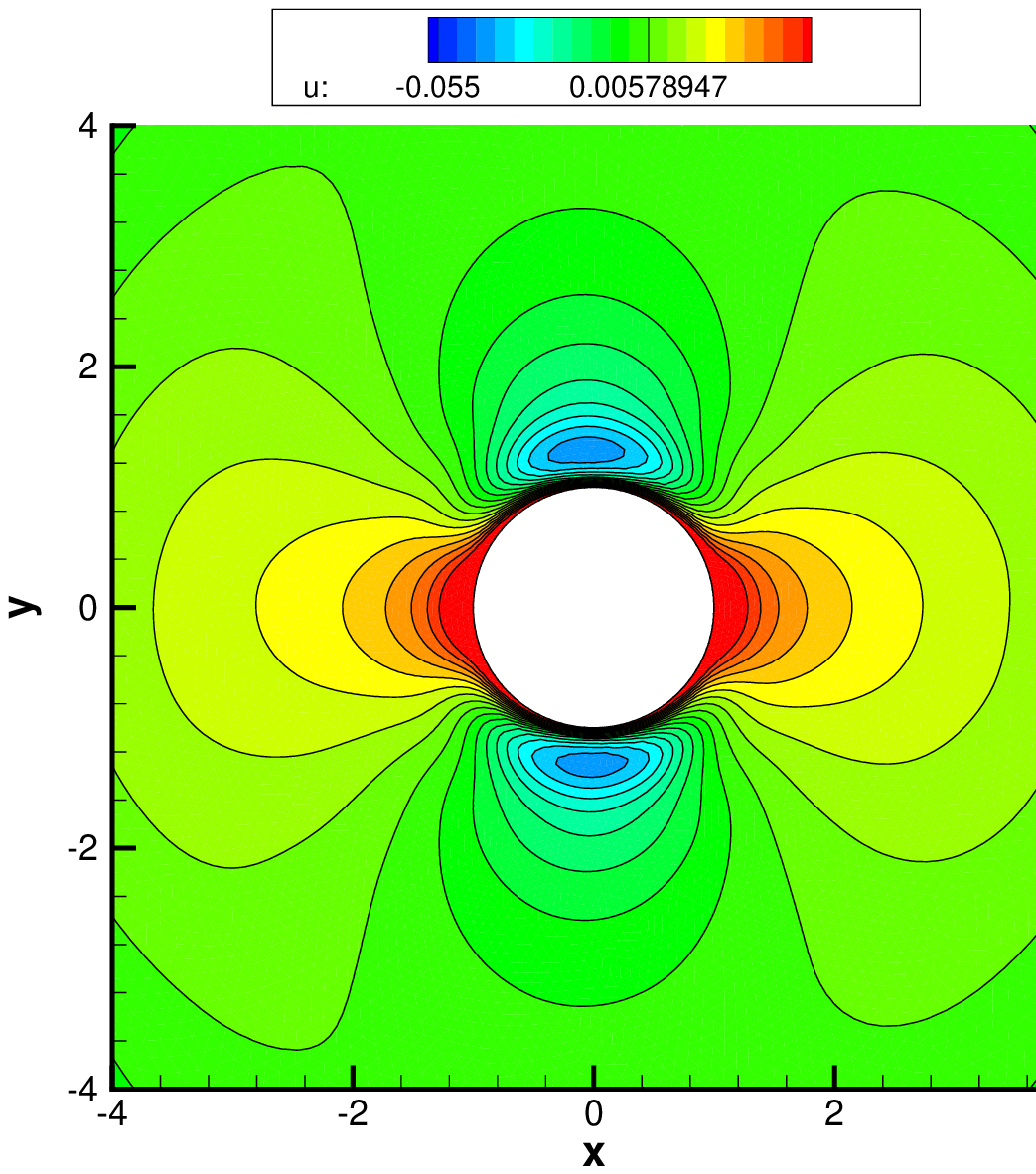}}
  \subfigure[$t=5$ s]{\includegraphics[width=0.4\textwidth]{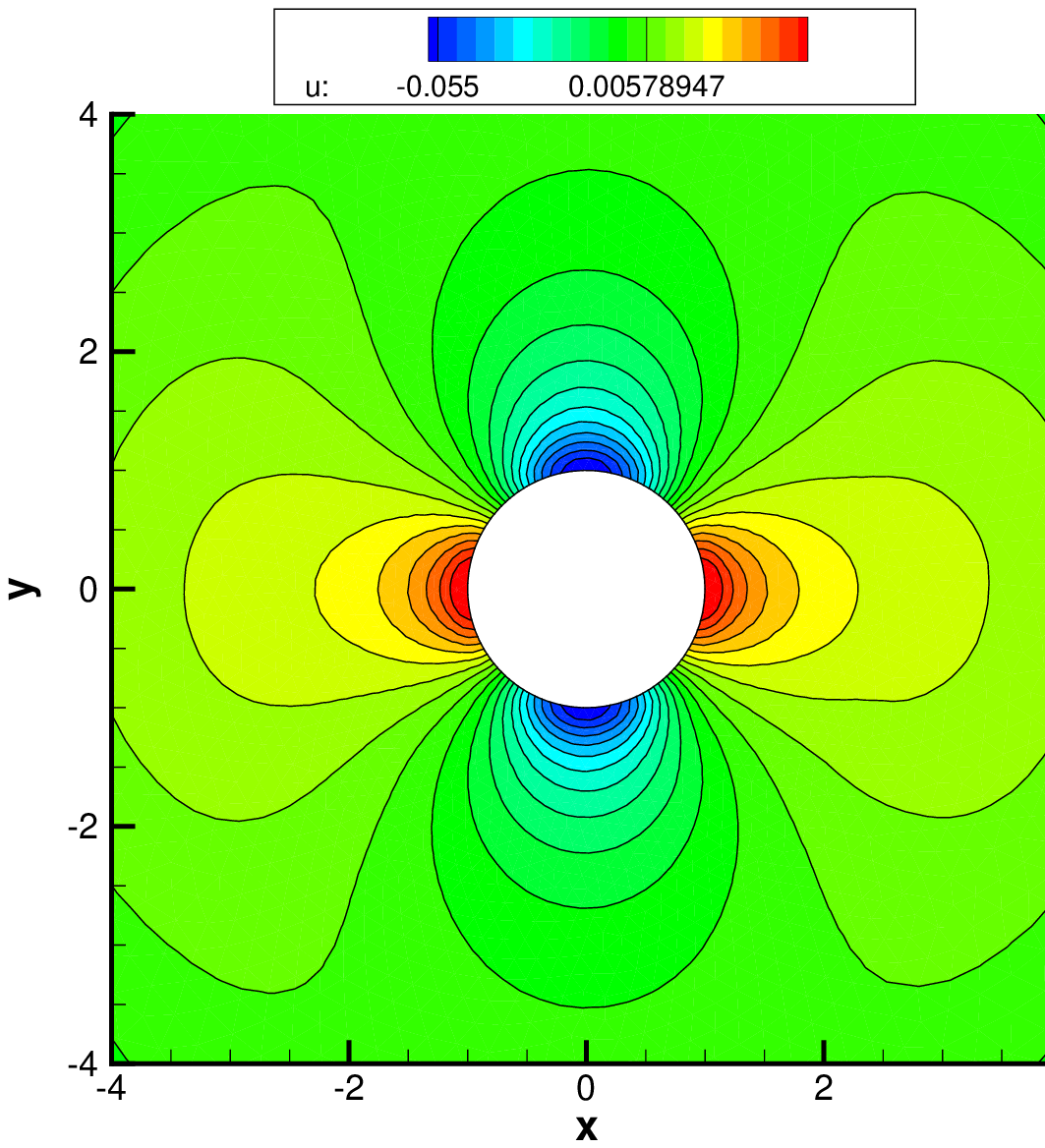}}
  \subfigure[$t=10$ s]{\includegraphics[width=0.4\textwidth]{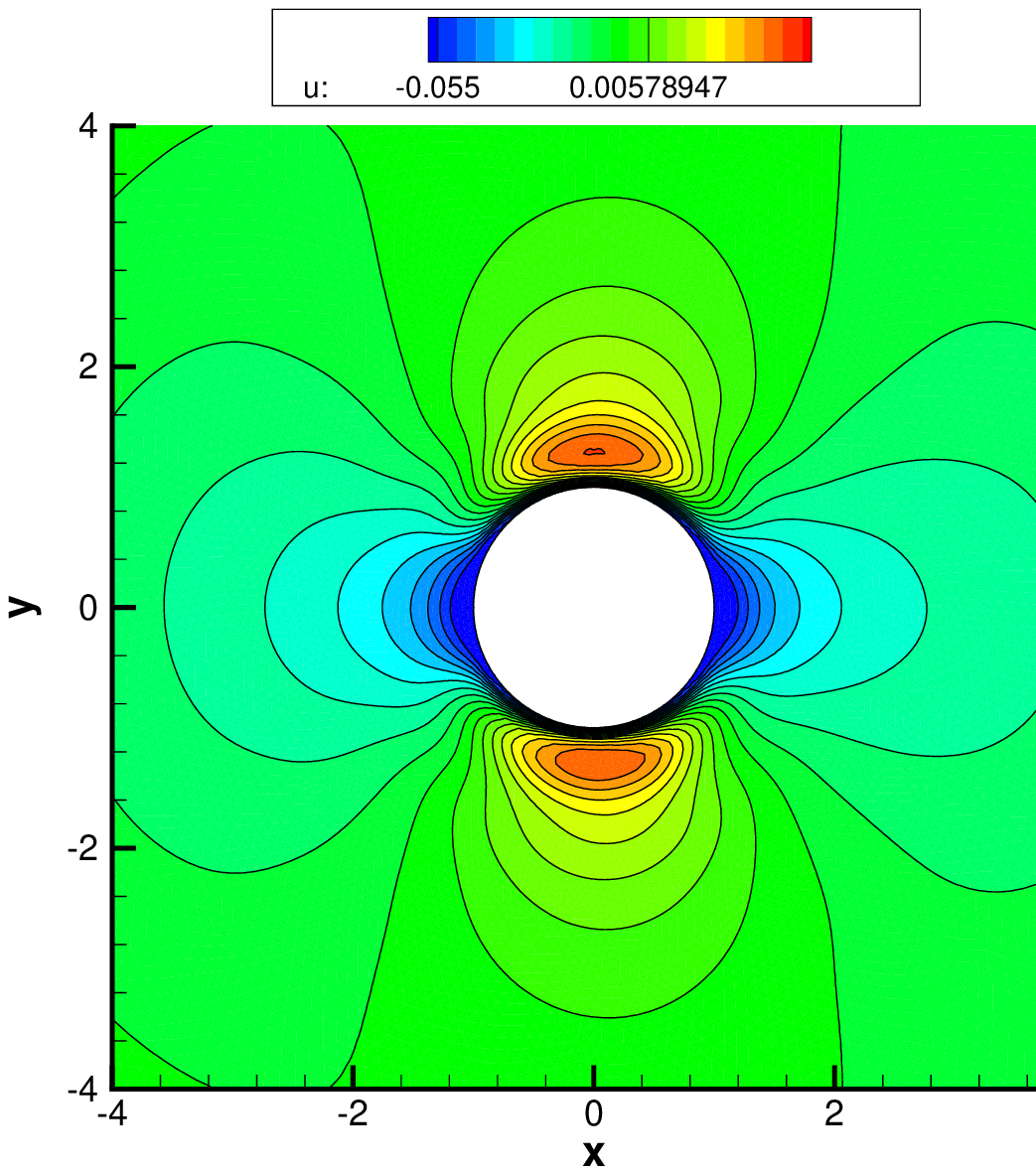}}
  \subfigure[$t=10$ s]{\includegraphics[width=0.4\textwidth]{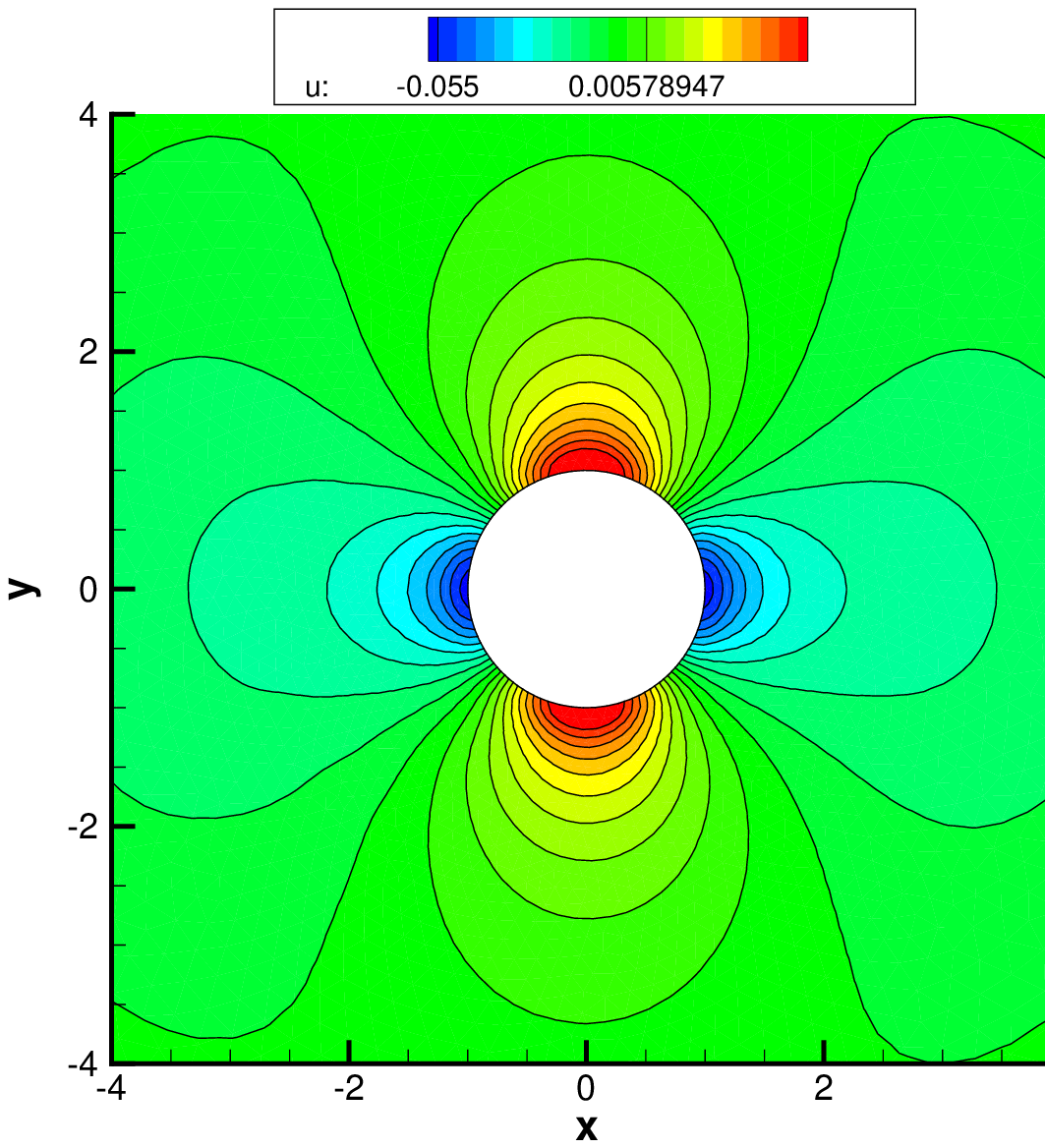}}
  \caption{Horizontal oscillating cylinder: velocity component $u$ distribution obtained with a FV-$\mathbb{P}_2$ method without correction (left) and with the ROD-$L^2$ correction (right).
  The solution is plotted at different times: $t=0.5$ s (top), $t=5$ s (middle) and $t=10$ s (bottom).}
  \label{fig:OC-H-velU}
\end{figure}

\begin{figure}
  \centering
  \subfigure[$t=0.5$ s]{\includegraphics[width=0.4\textwidth]{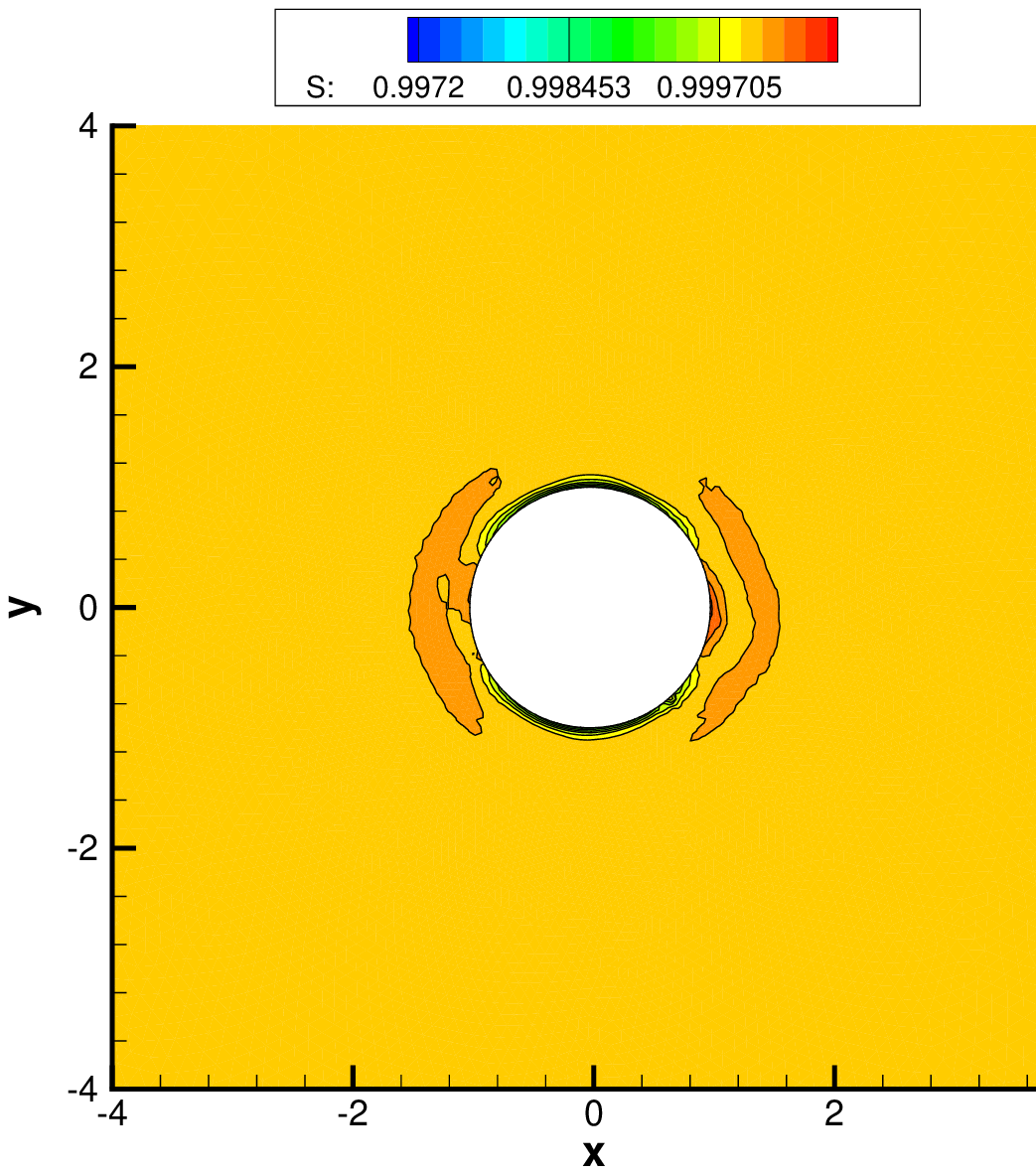}}
  \subfigure[$t=0.5$ s]{\includegraphics[width=0.4\textwidth]{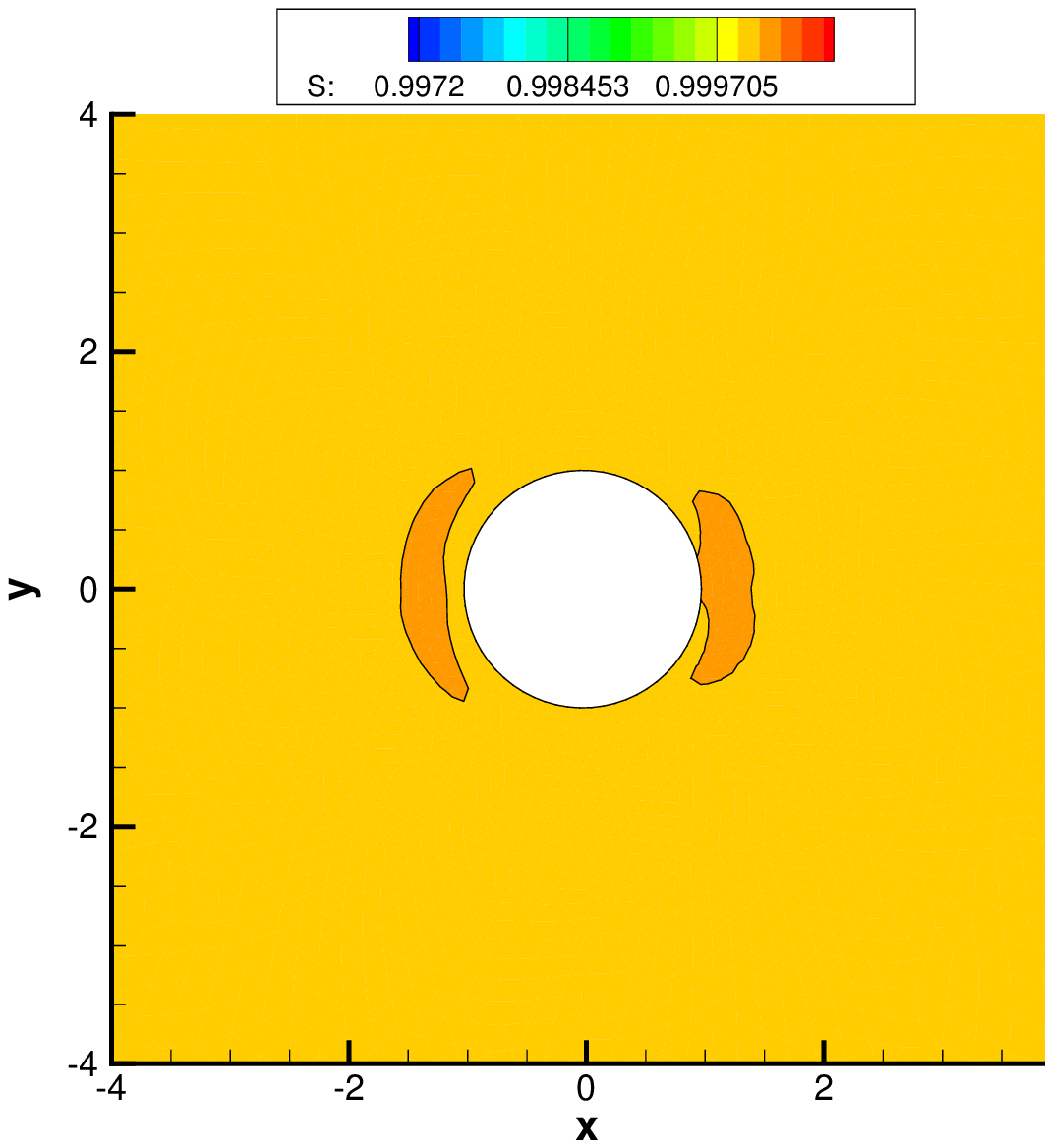}}
  \subfigure[$t=5$ s]{\includegraphics[width=0.4\textwidth]{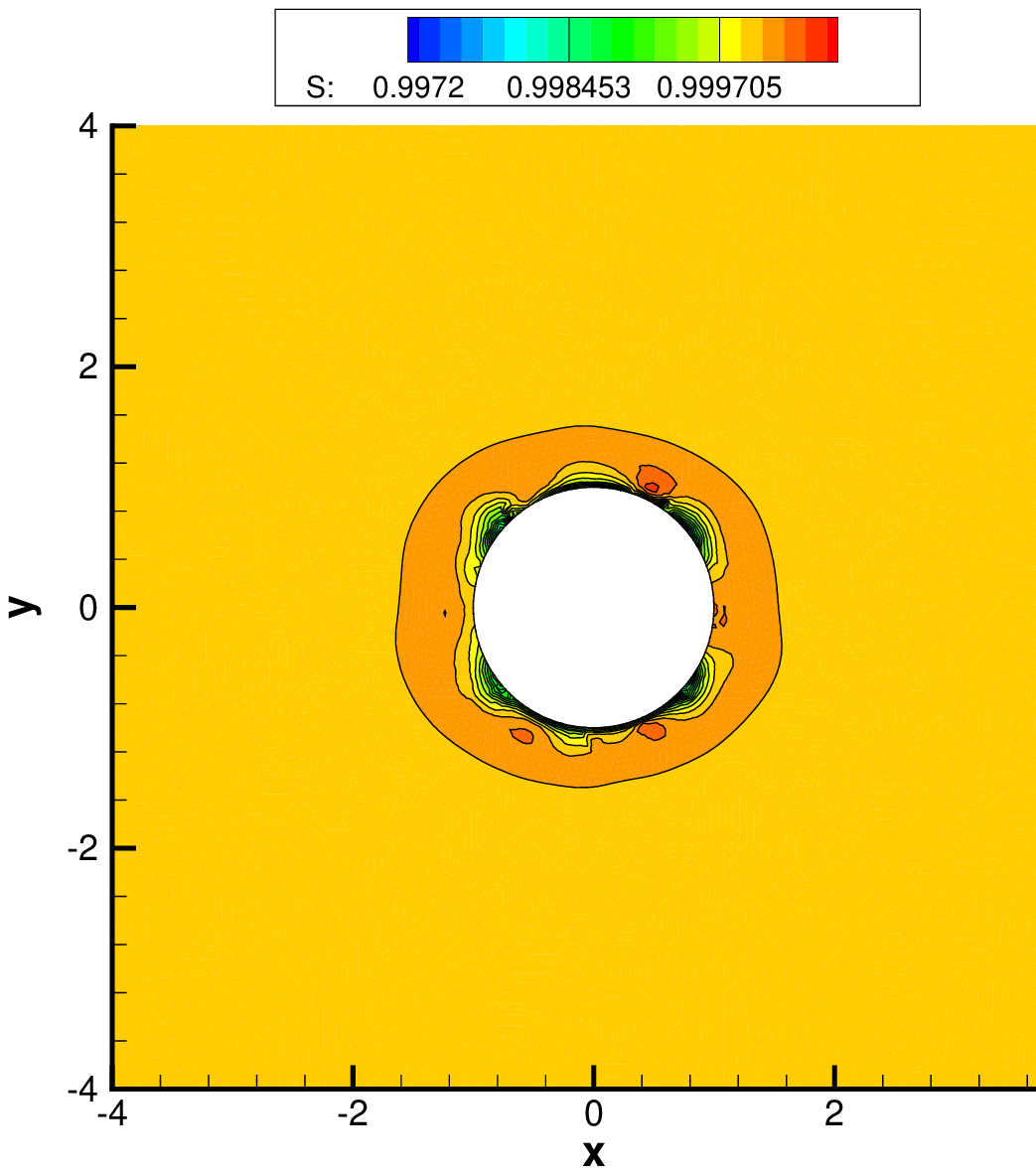}}
  \subfigure[$t=5$ s]{\includegraphics[width=0.4\textwidth]{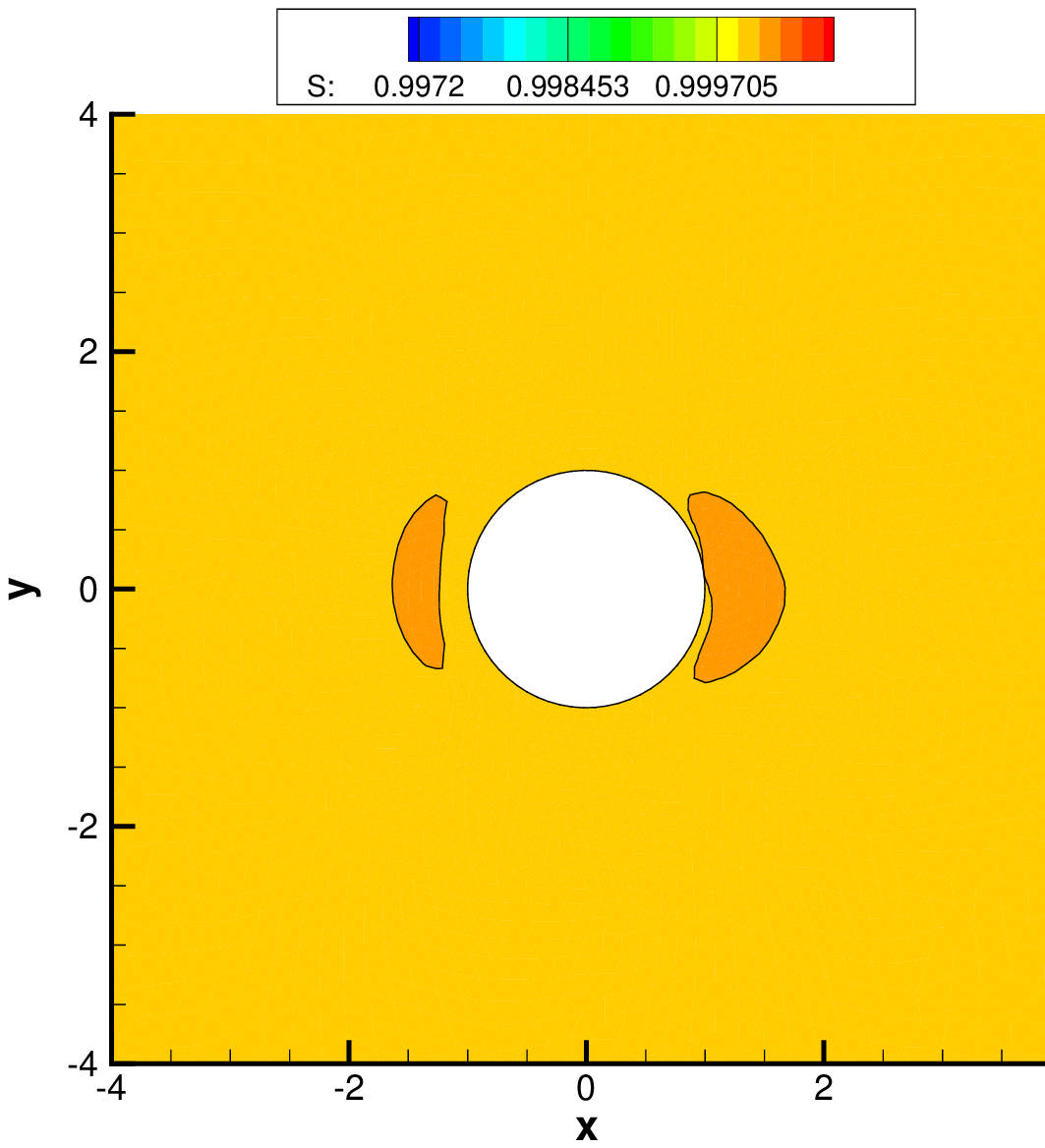}}
  \subfigure[$t=10$ s]{\includegraphics[width=0.4\textwidth]{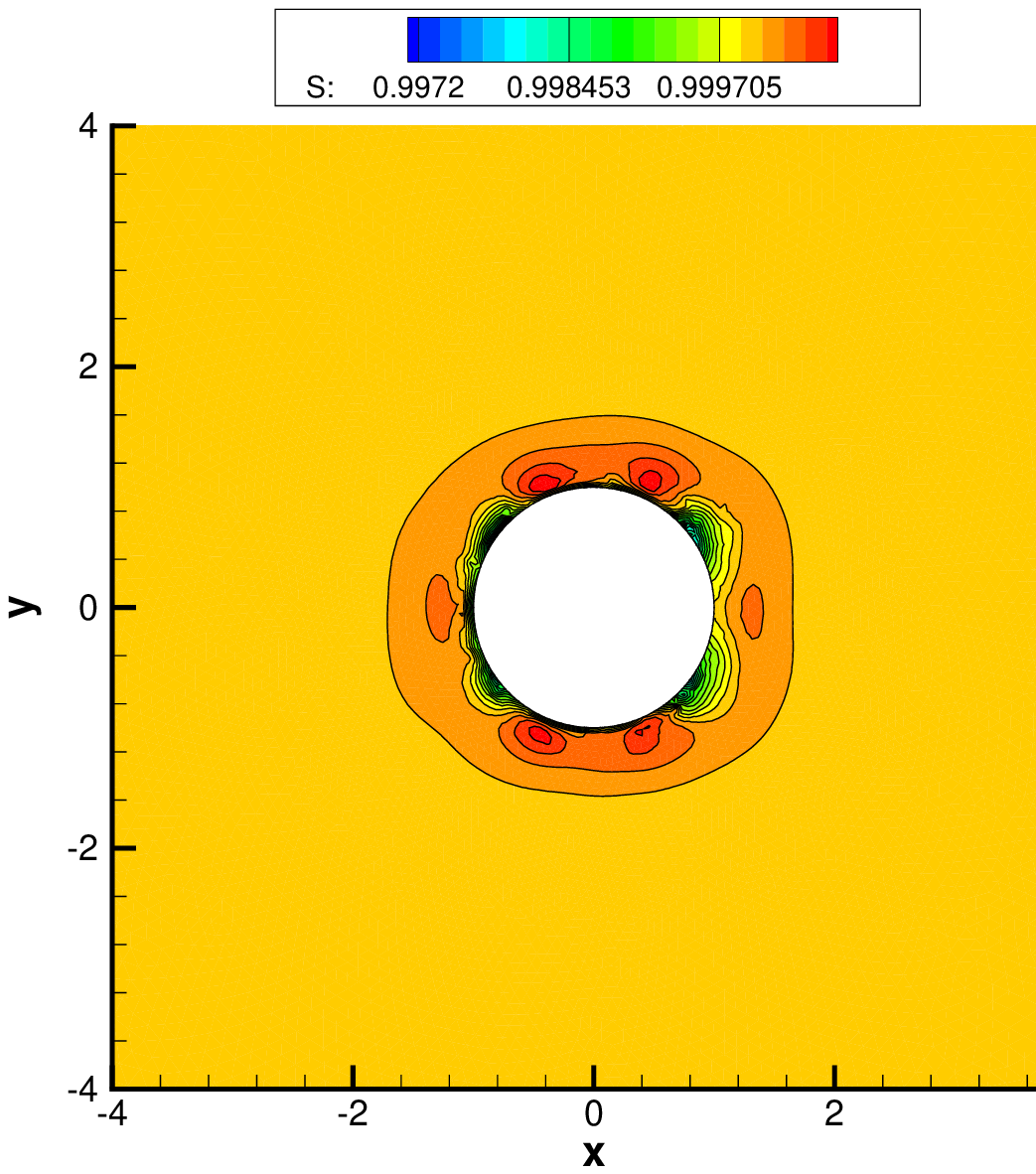}}
  \subfigure[$t=10$ s]{\includegraphics[width=0.4\textwidth]{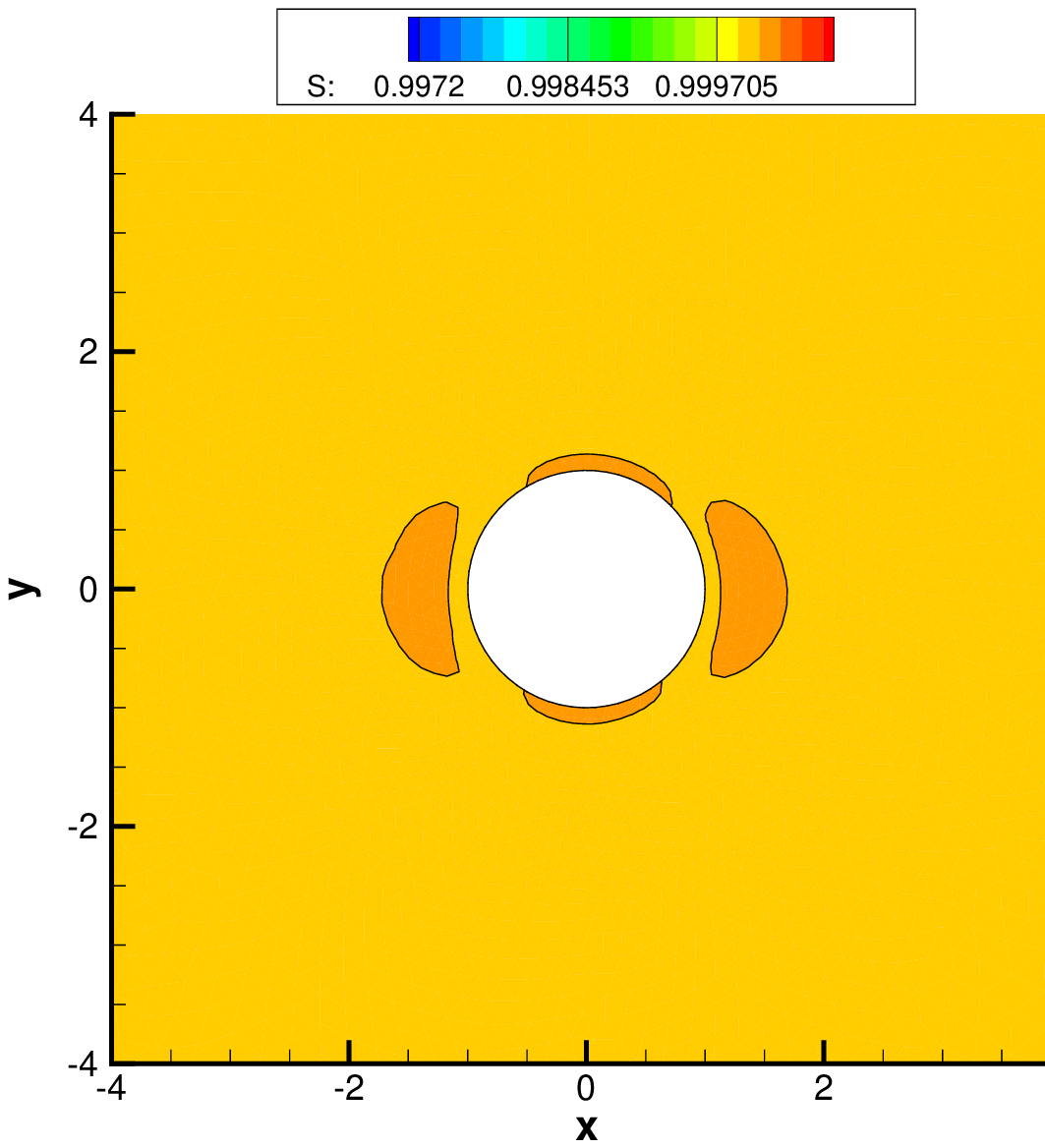}}
  \caption{Horizontal oscillating cylinder: entropy $S$ distribution obtained with a FV-$\mathbb{P}_2$ method without correction (left) and with the ROD-$L^2$ correction (right).
  The solution is plotted at different times: $t=0.5$ s (top), $t=5$ s (middle) and $t=10$ s (bottom).}
  \label{fig:OC-H-entropy}
\end{figure}

\section{Conclusions and future perspectives}\label{section:Conclusions}

In this work, we presented two novel high-order boundary treatments to compensate the geometrical error given by the piecewise affine approximation of curved boundaries. Their derivation is based on the manipulation of the ROD approach developed for DG in \cite{santos2024very} in a simplified one-dimensional setting. In this context, the ROD minimization problem no longer involves the inversion of the constraint matrix to build the modified polynomial. This allows one to manipulate it and derive a straightforward polynomial correction only valid for one given boundary point. The first contribution of this work is to show that this one-dimensional approach can be directly extended to multiple dimensions by applying it in the direction of the considered mapping $\mathcal{M}$, while preserving high-order convergence properties. In practice, the former ROD approach in multiple dimensions was building only one modified polynomial for the whole boundary cell. While the novel ROD-E and ROD-$L^2$ corrections build modified polynomials for every considered boundary point. This simplifies significantly the implementation and reduces the computational costs, since no matrix has to be inverted. The second contribution demonstrates that, although the analysis and derivation of the polynomial corrections have been formulated in the DG context, the same polynomial corrections can be applied in a versatile and general manner to several computational frameworks as long as a polynomial representation of the internal solution is provided.     
We prove the flexibility of these approaches by performing simulations and convergence analysis in the RK-DG framework with fixed curved domains, and in the more challenging ADER-ALE FV framework with moving curved domains. We show numerical results in both 2D and 3D with accuracy up to fifth order. 

The perspectives for these minimization-based methods are several. Given the novel established framework, new boundary treatments could be developed to preserve additional constraints or properties of the considered physical model. This would make them compatible with structure-preserving discretizations \cite{gassner2013skew}. More general high-order boundary conditions, like Neumann and Robin, could be developed for other physical models \cite{visbech2025spectral}. 
These approaches could be adapted to simulate high-order fluid-structure interaction problems \cite{gao2023high}, which for the moment rely on complex 3D curvilinear mesh motion \cite{boscheri2016high}. Additional studies will focus on the impact of these boundary treatments on boundary layers of viscous flows.

\section*{Acknowledgments}
WB acknowledges financial support from the \textit{Agence Nationale de la Recherche} (ANR, France) through the \textit{MONAIE} project and from the \textit{Institut de Mathématiques pour la Planète Terre} (IMPT, France) through the \textit{AAP2025-ERUPTA} project.
WB is member of the GNCS-INdAM (\textit{Istituto Nazionale di Alta Matematica}) group.

\bibliographystyle{elsarticle-num} 
\bibliography{literature}

\end{document}